\documentclass[12pt]{article}

\usepackage{amsfonts}
\usepackage{amssymb}
\usepackage{amsmath,amsthm}
\usepackage{latexsym}
\usepackage{amsbsy}
\usepackage{amscd}
\usepackage{color}

\newtheorem{teorema}{Theorem}[section]
\newtheorem{proposicion}[teorema]{Proposition}
\newtheorem{notacion}[teorema]{Notation}
\newtheorem{lema}[teorema]{Lemma}
\newtheorem{corolario}[teorema]{Corollary}
\newtheorem{ejemplo}[teorema]{Example}

\newtheorem{definicion}[teorema]{Definition}
\newtheorem{nota}[teorema]{Remark}

\def\bpm{B(\P^m,2)}
\def\fpm{F(\P^m,2)}

\def\TC{\mathrm{TC}}
\def\F2{\mathbb{F}_2}
\def\Z2{\mathbb{Z}_2}
\def\sq1{\mathrm{Sq}^1}

\def\ev{\mathrm{ev}}
\def\emb{\mathrm{Emb}}
\def\Imm{\mathrm{Imm}}

\def\P{\mathrm{P}}
\def\cajita{$\framebox{\begin{picture}(.7,.7)\put(0,0){}\end{picture}}
            \hspace{.2mm}$}

\title{The integral cohomology groups of configuration spaces
of pairs of points in real projective spaces}
\author{Jes\'us Gonz\'alez\hspace{.1mm}\footnote{Partially supported 
by CONACYT Research Grant 102783}\hspace{2.7mm} and \hspace{.1mm}
Peter Landweber}
\date{\empty}

\begin{document}

\maketitle

\begin{abstract}
We compute the integral homology and cohomology groups of 
configuration spaces of two distinct points on a given 
real projective space. The explicit answer is related
to the (known multiplicative structure in
the) integral cohomology---with simple and twisted 
coefficients---of the dihedral group of order~8 
(in the case of unordered configurations) 
and the elementary abelian 2-group of rank 2 (in the case of ordered 
configurations). As an application, we complete the computation 
of the symmetric topological complexity of real projective spaces
$\P^{2^i+\delta}$ with \hspace{.4mm}$i\geq0$\hspace{.4mm} and 
\hspace{.4mm}$0\leq\delta\leq2$.
\end{abstract}



\section{Introduction and description of main results}
Unless explicitly indicated otherwise, the notation 
$H^*(X)$ refers to integral cohomology groups of a space $X$ where 
a simple system of local coefficients is used. The cyclic group with $2^e$
elements is denoted by $\mathbb{Z}_{2^e}$. In the case $e=1$ we also use
the notation $\F2$ if the field structure is to be noted. It will be 
convenient to use the notation $\langle k\rangle$ for 
the elementary abelian 2-group of rank $k$, and write $\{k\}$ as a 
shorthand for $\langle k\rangle\oplus\mathbb{Z}_4$.

\medskip
We address the problem of computing the 
integral homology and cohomology groups of 
the configuration spaces $F(\P^m,2)$ and $B(\P^m,2)$ of two distinct points,
ordered and unordered respectively,
in the $m$-dimensional real projective space $\P^m$. 
Our main results are presented in Theorems~\ref{descripcionordenada}, 
\ref{descripciondesaordenada}, \ref{HFpar}, and \ref{HF1m4}.
The first two of these take the following explicit form:

\begin{teorema}\label{descripcionordenada}
For $n>0$,
$$
H^i(F(\P^{2n},2)) = \begin{cases}
\mathbb{Z}, & i=0\mbox{ ~or~ }i=4n-1;\\
\left\langle\frac{i}2+1\right\rangle, & i\mbox{ ~even,~ }1\leq i\leq2n;\\
\left\langle\frac{i-1}2\right\rangle, & i\mbox{ ~odd,~ }1\leq i\leq2n;\\
\left\langle2n+1-\frac{i}2\right\rangle, & i\mbox{ ~even,~ }
2n<i<4n-1;\rule{6mm}{0mm}\\
\left\langle2n-\frac{i+1}2\right\rangle, & i\mbox{ ~odd,~ }2n<i<4n-1;\\
0, & \mbox{otherwise}.
\end{cases}
$$
\noindent For $n\geq0$,
$$
H^i(F(\P^{2n+1},2)) = \begin{cases}
\mathbb{Z}, & i=0;\\
\left\langle\frac{i}2+1\right\rangle, & i\mbox{ ~even,~ }1\leq i\leq2n;\\
\left\langle\frac{i-1}2\right\rangle, & i\mbox{ ~odd,~ }1\leq i\leq2n;\\
\mathbb{Z}\oplus\langle n\rangle, & i=2n+1;\\
\left\langle2n+1-\frac{i}2\right\rangle, & i\mbox{ ~even,~ }2n+1<i\leq4n+1;\\
\left\langle2n+1-\frac{i-1}2\right\rangle, & i\mbox{ ~odd,~ }2n+1<i\leq4n+1;\\
0, & \mbox{otherwise}.
\end{cases}
$$
\end{teorema}

\begin{teorema}\label{descripciondesaordenada}
Let $0\leq b\leq3$. For $n>0$,
$$
H^{4a+b}(B(\P^{2n},2)) = \begin{cases}
\mathbb{Z}, & 4a+b=0 \mbox{ ~or~ } 4a+b=4n-1;\\
\{2a\}, & b=0<a,\;\,4a+b\leq2n;\\
\left\langle2a\right\rangle, & b=1,\;\,4a+b\leq2n;\\
\left\langle2a+2\right\rangle, & b=2,\;\,4a+b\leq2n;\\
\left\langle2a+1\right\rangle, & b=3,\;\,4a+b\leq2n;\\
\{2n-2a\}, & b=0,\;\,2n<4a+b<4n-1;\\
\langle2n-2a-1\rangle, & b=1,\;\,2n<4a+b<4n-1;\\
\langle2n-2a\rangle, & b=2,\;\,2n<4a+b<4n-1;\\
\langle2n-2a-2\rangle, & b=3,\;\,2n<4a+b<4n-1;\\
0, & \mbox{otherwise}.
\end{cases}
$$
\noindent For $n\geq0$,
$$
H^{4a+b}(B(\P^{2n+1},2)) = \begin{cases}
\mathbb{Z}, & 4a+b=0;\\
\{2a\}, & b=0<a,\;\,4a+b<2n+1;\\
\left\langle2a\right\rangle, & b=1,\;\,4a+b<2n+1;\\
\left\langle2a+2\right\rangle, & b=2,\;\,4a+b<2n+1;\\
\left\langle2a+1\right\rangle, & b=3,\;\,4a+b<2n+1;\\
\mathbb{Z}\oplus\langle n\rangle, & 4a+b=2n+1;\\
\{2n-2a\}, & b=0,\;\,2n+1<4a+b\leq4n+1;\\
\langle2n+1-2a\rangle, & b=1,\;\,2n+1<4a+b\leq4n+1;\\
\langle2n-2a\rangle, & b\in\{2,3\},\;\,2n+1<4a+b\leq4n+1;\\
0, & \mbox{otherwise}.
\end{cases}
$$
\end{teorema}

Theorems~\ref{descripcionordenada} and~\ref{descripciondesaordenada}
can be coupled with the Universal Coefficient Theorem (UCT), 
expressing homology in 
terms of cohomology (e.g.~\cite[Theorem~56.1]{munkres}), in order to give 
explicit descriptions of the corresponding integral homology groups.
Another immediate consequence is 
that, together with Poincar\'e duality (in its not necessarily orientable 
version, cf.~\cite[Theorem~3H.6]{hatcher} or~\cite[Theorem~4.51]{ranicki}),
Theorems~\ref{descripcionordenada} and~\ref{descripciondesaordenada} 
give a corresponding explicit description 
of the $w_1$-twisted homology and cohomology groups of $\fpm$ and $\bpm$.
Details are given in Section~\ref{linkinsection}. 

\medskip
It is to be observed that Theorem~\ref{descripciondesaordenada}
fully extends cohomological calculations for $\bpm$
in~\cite{bausum} (using a different 
approach). Rather elaborate Bockstein spectral 
sequence considerations in that paper led Bausum to a description 
of a few of the cohomology 
groups in Theorem~\ref{descripciondesaordenada}---groups that are
close to the top 
cohomological dimension $2m-1$. In turn, this leads to a description of sets 
$\emb_e(\P^m)$ of isotopy classes of smooth embeddings of $\P^m$ in 
$\mathbb{R}^{2m-e}$ for low values of $e$ (as low as $e\leq2$). Similar results
were obtained by Larmore and Rigdon (note the implicit hypothesis
$m>3$ in~\cite[Section~4]{LR})\footnote{We thank Sadok Kallel for pointing out
the results in~\cite{bausum} and~\cite{LR}.}. 
More recently, Section~3 in~\cite{taylor} explains how results like 
Theorems~\ref{descripcionordenada} and~\ref{descripciondesaordenada} could 
potentially lead to new embedding-type information about projective 
spaces (Theorem~\ref{STC} and Remark~\ref{optimalidad} below are based on
such a viewpoint).

\medskip
Theorem~\ref{descripcionordenada} implies 
that the torsion in $H^*(\fpm)$ is annihilated by 2. This observation
and a standard argument using the transfer of the double cover $\fpm\to\bpm$ 
show\footnote{We thank Fred Cohen for pointing out his argument.} that
the torsion subgroup of $H^*(\bpm)$ is annihilated by 4. 
Theorem~\ref{descripciondesaordenada} then shows this is a sharp bound, as 
$H^*(\bpm)$ has 4-torsion elements in dimensions $4\ell$ for $0<\ell<m/2$.
Thus, Theorem~\ref{descripciondesaordenada} proves a 
recent conjecture of Fred Cohen claiming that $H^*(B(\P^m,2))$ has 
2-torsion (for $m>1$) and 4-torsion (for $m>2$), but no 8-torsion. 

\begin{nota}\label{m1}{\em
Note that, after inverting 2, both $\bpm$ and $\fpm$ are homology 
spheres. This assertion can be considered as a partial 
generalization of the fact that both $F(\P^1,2)$ and $B(\P^1,2)$ have
the homotopy type of a circle; for $B(\P^1,2)$ this follows from  
Lemma~\ref{inicio2Handel} and Example~\ref{V22} below, while the situation for
$F(\P^1,2)$ comes from the fact that $\P^1$ is a Lie group---so that 
$F(\P^1,2)$ is in fact diffeomorphic to $S^1\times(S^1-\{1\})$. In particular,
any product of positive dimensional classes
in either $H^*(F(\P^1,2))$ or $H^*(B(\P^1,2))$ is trivial. The 
trivial-product property also
holds for both $H^*(F(\P^2,2))$ and $H^*(B(\P^2,2))$ in view of the $\P^2$-case
in Theorems~\ref{descripcionordenada} and~\ref{descripciondesaordenada}. But 
for $m\geq3$ both $\fpm$ and $\bpm$ should have useful integral cohomology
rings (\cite{taylor} contains partial information on the case $m=3$, 
as well as an application along the lines of Theorem 1.4 below), and this 
motivates the considerations in the rest of this introductory section.
}\end{nota}

The results in this paper go a bit 
further than a plain computation of cohomology
groups. Our ultimate motivation comes from the possibility of deducing new
information on the Euclidean embedding dimension of projective spaces 
based on a good hold of the cohomology rings of the relevant 
configuration spaces. Explicitly, let
\begin{equation}\label{classify}
B(\P^m,2)\stackrel{u}\longrightarrow\P^\infty
\end{equation}
classify the obvious double cover $F(\P^m,2)\to B(\P^m,2)$. Then, 
with the seven possible exceptions\footnote{Remark~\ref{optimalidad} 
below observes that we can now rule out the first of these 
potential exceptions.} of $m$ explicitly described 
in~\cite[Equation~(8)]{taylor}, 
$\emb(\P^m)$---the dimension of the smallest Euclidean space in which $\P^m$ 
can be smoothly embedded---is characterized as the
smallest integer $n=n(m)$ such that the map in~(\ref{classify}) 
can be homotopy compressed into $\P^{n-1}$. Furthermore, 
the main result in~\cite{symmotion} asserts that,
without restriction on $m$, the number $n(m)$ agrees with 
Farber's symmetric topological complexity\footnote{As indicated 
in Definition~\ref{reducida} at the end of the paper, here we use the 
reduced version of Farber's $\TC^S$, i.e.~we choose to 
normalize the Schwarz genus of a product fibration $F\times B\to B$ 
to be $0$---not $1$.} of $\P^m$, $\TC^S(\P^m)$, 
an invariant based on the motion planning problem in robotics.
From such a viewpoint, a proper understanding of the 
multiplicative height of $u^*(z)$, where $z$ is the generator in 
$H^2(\P^\infty)$, gives lower bounds on the values that $n$ can 
take---potentially leading to new information on the 
embedding problem of real projective spaces. 
The idea actually goes back at least as far as~\cite{handel68}, where mod 2 
coefficients (and obstruction theory) are used. But the 
$\mathbb{Z}_4$ groups appearing in Theorem~\ref{descripciondesaordenada} seem
to carry finer information not yet explored\footnote{Compare with the 
situation in~\cite{electron} where the topological Borsuk problem for 
$\mathbb{R}^3$ is studied via Fadell-Husseini index theory.}. 
For instance, the strategy using integral coefficients has recently been 
exploited in~\cite{taylor} in order 
to compute $\TC^S(\mathrm{SO}(3))$---identifying it as the
unique obstruction in Goodwillie's embedding Taylor 
tower for $\P^3$. 

\medskip
As an application of the cohomological results 
in this paper, our next result completes the computation started 
in~\cite{symmotion} of the symmetric topological
complexity of projective spaces of the form $\P^{2^i+\delta}$ in the range
$i\geq0$ and $0\leq\delta\leq2$.

\begin{teorema}\label{STC}
$\TC^S(\P^5)=\TC^S(\P^6)=9$.
\end{teorema}

Section~\ref{STCP67} starts with a discussion exhibiting
the case of $\P^6$ as giving the unique exceptional numerical value for 
$\TC^S(\P^{2^i+\delta})$ in the range $i\geq0$ and $0\leq\delta\leq2$.

\begin{nota}\label{optimalidad}{\em
Since $\emb(\P^5)=9$ (\cite{hopf,mahowald}), the list 
in~\cite{taylor} of seven exceptional values of $m$ for which the equality 
$\emb(\P^m)=\TC^S(\P^m)$ {\it could\hspace{.5mm}} 
fail reduces now to $\{6,7,11,12,14,15\}$. 
Note that $6$ is the smallest $m$ for which $\emb(\P^m)$ is 
unknown: $\emb(\P^6)\in\{9,10,11\}$ is the best assertion known to date 
(\cite{ES, mahowald}). On the other hand, Theorem~\ref{STC} obviously implies 
$\TC^S(\P^7)\geq9$, improving by 1 the previously known best lower 
bound for $\TC^S(\P^7)$ noted in~\cite[Table~1]{taylor}. In fact, 
taking into account Rees' PL embedding $\P^7\subset
\mathbb{R}^{10}$ constructed in~\cite{rees}, the above considerations imply
that both $\TC^S(\P^7)$ and $\emb_{\mathrm{PL}}(\P^7)$ lie in $\{9,10\}$,
which contrasts with the best known assertion about the 
embedding dimension of $\P^7$, namely $\emb(\P^7)\in\{9,10,11,12\}$
(\cite{Ha,M64}). Despite the fact that the equality $\Imm(\P^m)=\TC(\P^m)$
actually has three exceptions (related to the Hopf invariant one problem),
the above observations lead us to think that the equality 
$\emb(\P^m)=\TC^S(\P^m)$ should actually hold for every $m$, 
at least if $\emb$ is interpreted as {\it topological} embedding 
dimension. From such a perspective, it would be highly desirable
to know whether $\P^6$ topologically embeds in $\mathbb{R}^9$. 
On the other hand, it does not seem likely that $\P^7$ could possibly embed in 
$\mathbb{R}^9$ (even topologically), and the techniques proving 
Theorem~\ref{STC} (using perhaps a cohomology theory better suited than 
singular cohomology) might allow us to formalize our intuition---we hope 
to come back to such a point elsewhere.
}\end{nota}

A profitable approach to the kind of applications in the
previous paragraphs comes from using Handel's observation 
that~(\ref{classify}) factors through the classifying
space of the dihedral group $D_8$. Namely,~(\ref{classify}) is homotopic to
the composite
\begin{equation}\label{classifydescompuesta}
B(\P^m,2)\stackrel{p}\longrightarrow BD_8
\stackrel{q}\longrightarrow\P^\infty
\end{equation}
where $p$ is specified in 
Notation~\ref{convenience} below, and $q$ is specified 
in Remark~\ref{efectom} at the end of Section~\ref{HBprelim}.
Now, not only are $H^*(\P^\infty)$ and $H^*(BD_8)$ well-known
rings, but the induced ring map 
$q^*$ is well understood (Remark~\ref{efectom}). But most importantly,
the induced map $p^*$ surjects onto the torsion subgroups of
$H^*(\bpm)$ except, perhaps, for $m\equiv3\bmod4$ 
(Theorems~\ref{HFpar} and~\ref{HF1m4} below). 
So, an eventual study of the multiplicative height of 
$u^*(z)$, and of the ring $H^*(\bpm)$ for that matter, can
be reduced to having a good hold on the kernel of $p^*$,
i.e., Fadell-Husseini's ideal-valued $\mathbb{Z}$-index 
of the $D_8$-restricted action of O$(2)$ on $V_{m+1,2}$---see 
Definition~\ref{inicio1Handel}, Lemma~\ref{inicio2Handel}, and
the considerations around~(\ref{FHindex}).
For the remainder of this 
section our attention focuses on the likely surjectivity
property of $p^*$ and, with this in mind, the following considerations (see for 
instance~\cite[\S2]{handel68}) are our main tool:

\begin{definicion}\label{inicio1Handel}{\em
Recall that $D_8$ can be expressed as the usual wreath 
product extension 
\begin{equation}\label{wreath}
1\to\Z2\times\Z2\to D_8\to\Z2\to1.
\end{equation}
Let $\rho_1,\rho_2\in D_8$ generate the normal subgroup 
$\Z2\times\Z2$, and let (the class of)
$\rho\in D_8$ generate the quotient group $\Z2$ so that, via conjugation,
$\rho$ switches $\rho_1$ and $\rho_2$. $D_8$ acts 
freely on the Stiefel
manifold $V_{n,2}$ of orthonormal $2$-frames in $\mathbb{R}^n$
by setting $$\rho(v_1,v_2)=(v_2,v_1),\quad
\rho_1(v_1,v_2)=(-v_1,v_2), \quad\mbox{and}\quad\rho_2(v_1,v_2)=(v_1,-v_2).$$
This describes a group inclusion $D_8\hookrightarrow\mathrm{O}(2)$
where the rotation 
$\rho\rho_1$ is a generator for $\mathbb{Z}_4=D_8\cap\mathrm{SO}(2)$.
}\end{definicion}

\begin{notacion}\label{convenience}{\em
Throughout the paper the letter $G$ stands for either 
$D_8$ or its subgroup $\Z2\times\Z2$ in~(\ref{wreath}). Likewise,
$E_m=E_{m,G}$ denotes the orbit space 
of the $G$-action on $V_{m+1,2}$ indicated in Definition~\ref{inicio1Handel},
and $\theta\colon V_{m+1,2}\to E_{m,G}$ represents the canonical projection.
As explained in the paragraph containing~(\ref{classifydescompuesta}), 
our interest lies in the (kernel of the) morphism induced in 
cohomology by the map
\begin{equation}\label{lap}
p=p_{m,G}\colon E_m\to BG 
\end{equation}
that classifies the $G$-action on $V_{m+1,2}$.
}\end{notacion}

\begin{lema}[{\cite[Proposition~2.6]{handel68}}]\label{inicio2Handel}
$E_m$ is a strong deformation retract of $B(\P^m,2)$ 
if $G=D_8$, and of $F(\P^m,2)$ if $G=\Z2\times\Z2$.\hfill\cajita
\end{lema}

Thus, the cohomology properties of the configuration spaces we are 
interested in---and of~(\ref{lap}), for that matter---can be approached
via the Cartan-Leray spectral sequence (CLSS) of the $G$-action on $V_{m+1,2}$.
Such an analysis yields:

\begin{teorema}\label{HFpar}
Let $m$ be even. The map $p^*\colon H^i(BG)\to H^i(E_m)$ is:
{\em\begin{enumerate}
\item {\em an isomorphism for $i\leq m;$}
\item {\em an epimorphism with nonzero kernel for $m<i<2m-1;$}
\item {\em the zero map for $2m-1\leq i$.}
\end{enumerate}}
\end{teorema}

\begin{teorema}\label{HF1m4}
Let $m$ be odd. The map $p^*\colon H^i(BG)\to H^i(E_m)$ is:
{\em \begin{enumerate}
\item {\em an isomorphism for $i<m;$}
\item {\em a monomorphism onto the 
torsion subgroup of $H^i(E_m)$ for $i=m;$}
\item {\em the zero map for $2m-1<i$.}
\end{enumerate}}
\noindent Further, $p^*$ is an epimorphism with nonzero kernel 
for $m<i\leq 2m-1$ except perhaps when $G=D_8$ and $\,m\equiv3\bmod4$.
\end{teorema}

Since the ring $H^*(BG)$ is well known 
(see Theorem~\ref{modulestructure} and the comments following 
Lemma~\ref{kunneth}), the multiplicative 
structure of $H^*(E_m)$ through dimensions at most $m$ 
follows from the four theorems stated in this section. Furthermore, 
much of the ring structure in larger dimensions 
now depends on giving explicit generators for the ideal Ker$(p^*)$. 
In this direction we prove the following result (noticed independently 
by Fred Cohen using different methods):

\begin{proposicion}\label{mono4}
Let $G=D_8$. Assume $m\not\equiv3\bmod4$ and consider the 
map in~{\em(\ref{lap})}. In dimensions at most $2m-1$, every
nonzero element in $\mathrm{Ker}(p^*)$ has order $2$, 
i.e.~$2\cdot\mathrm{Ker}(p^*)=0$ in those dimensions. 
In fact, every $4\ell$-dimensional integral cohomology class in $BD_8$ 
generating a $\mathbb{Z}_4$-group maps under 
$p^*$ into a class which also generates a $\mathbb{Z}_4$-group provided 
$\ell<m/2$---otherwise the class maps trivially for dimensional reasons.
\end{proposicion}

\begin{nota}\label{F2dimension}{\em
By Lemma~\ref{kunneth} below, 
$\mathrm{Ker}(p^*)$ is also killed by multiplication by 2
when $G=\Z2\times \Z2$ (any $m$, any dimension).
Our approach allows us to explicitly describe
the (dimension-wise) 2-rank of $\mathrm{Ker}(p^*)$ in the cases where we 
know this is an $\F2$-vector space (i.e.~when either $G=\Z2\times\Z2$ or
$m\not\equiv3\bmod4$, 
see Examples~\ref{muchkernel1} and~\ref{muchkernel2}). 
Unfortunately the methods used in the proofs of Proposition~\ref{mono4}
and Theorems~\ref{HFpar} and~\ref{HF1m4} break down for $E_{4n+3,D_8}$,
and Section~\ref{problems3} in the preliminary 
version~\cite{v1} of this paper
discusses a few such aspects, mainly focusing attention on the case $n=0$.
We hope this paper serves as a motivation to study the case of 
$B(\P^{4n+3},2)$ in Theorem~\ref{HF1m4}, as well as to get a hold on the 
complete ring structure of $H^*(E_m)$ or, for that matter, on the kernel 
of $p^*$---aiming, for instance, at the 
geometric applications sketched in the paragraph containing~(\ref{classify}).
}\end{nota}

The spectral sequence methods in this paper are similar 
in spirit to those in~\cite{idealvalued}
and~\cite{FZ}. In the latter reference, Feichtner and Ziegler
describe the integral cohomology rings of {\it ordered} configuration spaces
on spheres by means of a full analysis of the Serre
spectral sequence (SSS) associated
to the Fadell-Neuwirth fibration $\pi\colon F(S^k,n)\to S^k$ given by
$\pi(x_1,\ldots, x_n)=x_n$ (a similar study is carried out
in~\cite{FZ02}, but in the context of {\em 
ordered} orbit configuration spaces). 
One of the main achievements of the present
paper is a successful calculation of cohomology groups  
of {\it unordered} configuration spaces (on real projective spaces),
where no Fadell-Neuwirth 
fibrations are available---instead we rely on
Lemma~\ref{inicio2Handel} and the CLSS\footnote{Our 
CLSS calculations can also be done in terms of the SSS of the fibration
$V_{m+1,2}\stackrel{\theta}\to E_{m,G}\stackrel{p}{\to} BG$.} 
of the $G$-action on $V_{m+1,2}$. Also worth stressing is the fact 
that we succeed in computing cohomology groups with {\it integer} coefficients,
whereas the Leray spectral sequence (and its $\Sigma_k$-invariant version) 
for the inclusion $F(X,k)\hookrightarrow X^k$ has proved to be 
effectively computable mainly when {\it field} coefficients are used
(\cite{FeTa,Totaro}).

\medskip
A major obstacle we have to confront (not 
present in~\cite{FZ}) comes from the fact that the spectral sequences
we encounter often have non-simple systems of local coefficients.
This is also the situation in~\cite{idealvalued}, where 
the two-hyperplane case of Gr\"unbaum's mass partition 
problem~(\cite{grunbaum}) is studied from the Fadell-Husseini 
index theory viewpoint~\cite{FH}. Indeed, Blagojevi\'c and Ziegler
deal with twisted coefficients in their main SSS, namely the one associated to 
the Borel fibration
\begin{equation}\label{FHindex}
S^m\times S^m \to ED_8\times_{D_8}(S^m\times S^m)
\stackrel{\overline{p}}\to BD_8
\end{equation}
where the $D_8$-action on $S^m\times S^m$ is the obvious extension of 
that in Definition~\ref{inicio1Handel}. Now, the main goal in~\cite{idealvalued}
is to describe the kernel of the map induced by $\overline{p}$ in integral 
cohomology---the so-called Fadell-Husseini ($\mathbb{Z}$-)index of $D_8$
acting on $S^m\times S^m$, Index$_{D_8}(S^m\times S^m)$. 
Since $D_8$ acts freely on $V_{m+1,2}$, Index$_{D_8}(S^m\times 
S^m)$ is contained in the kernel of the map induced in integral
cohomology by the map $p\colon E_m\to BD_8$
in Proposition~\ref{mono4} (whether or not $m\equiv3\bmod4$). 
In particular, the work in~\cite{idealvalued} can be used to
identify explicit elements in Ker$(p^*)$ and, as observed
in Remark~\ref{F2dimension}, our approach allows us to assess, for $m\not
\equiv3\bmod4$ (in Examples~\ref{muchkernel1} and \ref{muchkernel2}), how
much of the actual kernel is still lacking 
description: \cite{idealvalued} gives just a bit less than half 
the expected elements in Ker$(p^*)$.

\section{Preliminary cohomology facts}\label{HBprelim}
As shown {in~\cite{AM} (see also~\cite{handel68}
for a straightforward approach)}, the mod 2 cohomology of $D_8$ is 
a polynomial ring on three generators $x,x_1,x_2\in H^*(BD_8;\F2)$, 
the first two of dimension 1, and the last one of dimension 2,
subject to the single relation $x^2=x\cdot x_1$. The classes $x_i$ are
the restrictions of the universal Stiefel-Whitney classes $w_i$
($i=1,2$) under the map corresponding to the group inclusion
$D_8\subset\mathrm{O}(2)$ in Definition~\ref{inicio1Handel}. 
On the other hand, the class $x$ is not characterized by the relation
$x^2=x\cdot x_1$, but by the requirement that, for all $m$, $x$ 
pulls back to~(\ref{classify}) under the map $p_{m,D_8}$ 
in~(\ref{lap})---see~\cite[Proposition~3.5]{handel68}. 
In particular:

\begin{lema}\label{handel68D}
For $i\geq0$, $H^i(BD_8;{\F2})=\langle i+1\rangle$.\hfill\cajita
\end{lema}

\begin{corolario}\label{handel68B}
For any $m$, $$H^i(B(\P^m,2);{\F2})=\begin{cases}\langle 
i+1\rangle,& 0\leq i\leq m-1;\\
\langle2m-i\rangle,& m\leq i\leq 2m-1;\\0,&\mbox{otherwise}.\end{cases}$$
\end{corolario}
\begin{proof}
The assertion for $i\geq2m$ follows from Lemma~\ref{inicio2Handel} and 
dimensional considerations.
Poincar\'e duality implies that
the assertion for $m\leq i\leq 2m-1$ follows from that for $0\leq i\leq m-1$.
Since $V_{m+1,2}$ is ($m-2$)-connected, 
the assertion for $0\leq i\leq m-1$ follows from Lemma~\ref{handel68D},
using the fact (a consequence 
of~\cite[Proposition~3.6 and~(3.8)]{handel68})
that, in the mod 2 SSS for the fibration $V_{m+1,2}\stackrel{\theta}{\to}
E_{m,D_8}\stackrel{p}{\to} BD_8$, 
the two indecomposable elements in $H^*(V_{m+1,2};\F2)$
transgress to nontrivial elements.
\end{proof}

Let $\mathbb{Z}_\alpha$ denote the $\mathbb{Z}[D_8]$-module 
whose underlying group is free on a generator $\alpha$ on which 
each of $\rho,\rho_1,\rho_2\in D_8$ acts via multiplication
by $-1$ (in particular, elements in $D_8\cap\mathrm{SO}(2)$ act 
trivially). Corollaries~\ref{HBD} and~\ref{HBDT} below 
are direct consequences of the following description, proved 
in~\cite{handeltohoku} (see also~\cite[Theorem~4.5]{idealvalued}),
of the ring $H^*(BD_8)$ and of the $H^*(BD_8)$-module $H^*(BD_8;
\mathbb{Z}_\alpha)$:

\begin{teorema}[Handel~\cite{handeltohoku}]\label{modulestructure}
$H^*(BD_8)$ is generated by classes $\mu_2$, $\nu_2$, $\lambda_3$, 
and $\kappa_4$ subject to the relations
$2\mu_2=2\nu_2=2\lambda_3=4\kappa_4=0$, $\nu_2^2=\mu_2\nu_2$, and 
$\lambda_3^2=\mu_2\kappa_4$.
$H^*(BD_8;\mathbb{Z}_\alpha)$ is the free $H^*(BD_8)$-module on 
classes $\alpha_1$ and $\alpha_2$ subject to the relations 
$2\alpha_1=4\alpha_2=0$, $\lambda_3\alpha_1=\mu_2\alpha_2$, and 
$\kappa_4\alpha_1=\lambda_3\alpha_2$. Subscripts in the notation of these 
six generators indicate their cohomology dimensions.\hfill\cajita
\end{teorema}

The notation $a_2$, $b_2$, $c_3$, and $d_4$ 
was used in~\cite{handeltohoku}
instead of the current $\mu_2$, $\nu_2$, $\lambda_3$, and $\kappa_4$. 
The change is made in order to avoid confusion with the generic notation $d_i$ 
for differentials in the several spectral sequences considered
in this paper.

\begin{corolario}\label{HBD}
For $a\geq0$ and $0\leq b\leq3$,
$$
H^{4a+b}(BD_8)=\begin{cases}
\mathbb{Z}, & (a,b)=(0,0);\\
\{2a\}, & b=0<a;\\
\langle 2a\rangle, & b=1;\\
\langle 2a+2\rangle, & b=2;\\
\langle 2a+1\rangle, & b=3.
\end{cases}
$$

\vspace{-7mm}\ \hfill\cajita
\end{corolario}

\begin{corolario}\label{HBDT}
{For $a\geq0$ and $0\leq b\leq3$,}
$$
{H^{4a+b}(BD_8;\mathbb{Z}_\alpha)=\begin{cases}
\langle 2a\rangle, & b=0;\\
\langle 2a+1\rangle, & b=1;\\
\{2a\},& b=2;\\
\langle 2a+2\rangle, & b=3.
\end{cases}}
$$

\vspace{-7mm}\ \hfill\cajita
\end{corolario}

We show that, up to a certain symmetry condition (exemplified in
Table~\ref{tabla} at the end of Section~\ref{linkinsection}), 
the groups explicitly described 
by Corollaries~\ref{HBD} and~\ref{HBDT} delineate the 
additive structure of the graded group $H^*(B(\P^m,2))$. 
The corresponding situation for $H^*(F(\P^m,2))$ uses the 
following well-known analogues of Lemma~\ref{handel68D} and
Corollaries~\ref{handel68B},~\ref{HBD} and~\ref{HBDT}:

\begin{lema}\label{wellknown}
For $i\geq0$, $H^i(\P^\infty\times\P^\infty;\F2)=\langle i+1\rangle$.
\hfill\cajita
\end{lema}

\begin{lema}\label{initialF}
For any $m$, $$H^i(F(\P^m,2);{\F2})=\begin{cases}\langle 
i+1\rangle,& 0\leq i\leq m-1;\\
\langle2m-i\rangle,& m\leq i\leq 2m-1;\\0,&\mbox{otherwise}.\end{cases}$$

\vspace{-8mm}\ \hfill\cajita
\end{lema}

\begin{lema}\label{kunneth} For $i\geq0$,
\begin{eqnarray*}
H^i(\P^\infty\times\P^\infty)&=&\begin{cases}
\mathbb{Z}, & i=0; \\ \left\langle\frac{i}2+1\right\rangle, & 
i \mbox{ even }, i>0; \\ \left\langle\frac{i-1}2\right\rangle, 
& \mbox{otherwise}.
\end{cases} \\ 
H^i(\P^\infty\times\P^\infty;\mathbb{Z}_\alpha)&=&\begin{cases} 
\left\langle\frac{i}{2}\right\rangle, & i \mbox{ even};\\ \left\langle
\frac{i+1}{2}\right\rangle, & i \mbox{ odd}. \end{cases}
\end{eqnarray*}
Here $\mathbb{Z}_\alpha$
is regarded as a $(\Z2\times\Z2)$-module via the restricted structure 
coming from the inclusion $\Z2\times\Z2\hookrightarrow D_8$.
\hfill\cajita
\end{lema}

Here are some brief comments on the proofs of 
Lemmas~\ref{wellknown}--\ref{kunneth}.
Of course, the ring structure $H^*(\P^\infty\times\P^\infty;\F2)=
\F2[x_1,y_1]$ is standard (as in Theorem~\ref{modulestructure}, subscripts 
for the cohomology classes in this paragraph indicate dimension).
On the other hand, it is easily shown~(see for instance~\cite[Example~3E.5 
on pages~306--307]{hatcher}) that $H^*(\P^\infty\times\P^\infty)$
is the polynomial ring over the integers on three classes $x_2$, $y_2$,
and $z_3$ subject to the four relations
\begin{equation}\label{relacionesenteras}
2x_2=0,\;\; 2y_2=0,\;\; 2z_3=0,\;\mbox{and}\;\; z_3^2=x_2y_2(x_2+y_2).
\end{equation} 
These two facts yield Lemma~\ref{wellknown} and
the first equality in Lemma~\ref{kunneth}. 
Lemma~\ref{initialF} can be proved with the argument given for
Corollary~\ref{handel68B}---replacing $D_8$ by its subgroup 
$\Z2\times\Z2$ in~(\ref{wreath}).
Finally, both equalities in Lemma~\ref{kunneth} can be obtained as immediate
consequences of the K\"unneth exact sequence (for the second 
equality, note that $\mathbb{Z}_\alpha$ arises
as the tensor square of the standard twisted coefficients for a single 
factor $\P^\infty$). 

\begin{nota}\label{mapdereduccion}{\em
For future reference we recall (again from Hatcher's 
book) that the mod 2 reduction map $H^*(\P^\infty\times\P^\infty)\to 
H^*(\P^\infty\times\P^\infty;\F2)$, a monomorphism in positive dimensions,
is characterized by $x_2\mapsto x_1^2$,
$y_2\mapsto y_1^2$, and $z_3\mapsto x_1y_1(x_1+y_1)$.
}\end{nota}

\begin{nota}\label{efectom}{\em
Here are the promised details about the factorization of~(\ref{classify})
through $BD_8$. We already noticed that the claimed 
factorization~(\ref{classifydescompuesta}) is proved 
in~\cite{handel68}---for $m\geq 3$, but the restriction can be removed by 
naturality---where 
$q\colon BD_8\to\P^\infty$ corresponds to the class $x\in H^1(BD_8;\F2)$ at 
the beginning of the section. On the other hand, the 
extension~(\ref{wreath}) defines a fibration
$$\P^\infty\times\P^\infty\stackrel{\iota}\longrightarrow 
BD_8\stackrel{q'}\longrightarrow\P^\infty,$$
and Handel's proof of~\cite[Proposition~3.5]{handel68} characterizes
$x$ as the only nonzero element in $H^1(BD_8;\F2)$ 
mapping trivially under the fiber inclusion $\iota$. Thus, in fact $q=q'$.
In particular, the map induced by $q$ in 
integral cohomology can be computed 
in purely algebraic terms, using the projection in~(\ref{wreath}).
Actually, since $H^*(\P^\infty)=\mathbb{Z}[z]\left/(2z)\right.$ 
with $z\in H^2(\P^\infty)$, $q^*$ is determined by its 
value on $z$. As the reader can easily verify,
a simple exercise using the well-known resolution of
the (trivial) $D_8$-module $\mathbb{Z}$ (see for instance~\cite{handeltohoku})
shows that generators in Theorem~\ref{modulestructure} can be chosen 
so that $q^*(z)=\nu_2$.
}\end{nota}

\section{Orientability properties of some quotients of $V_{n,2}$}
\label{orientability}

Proofs in this section will be postponed 
until all relevant results have been presented. 
Recall that all Stiefel manifolds $V_{n,2}$ are orientable (actually 
parallelizable, cf.~\cite{sutherland}).
Even if some of the elements of a given subgroup $H$ 
of $\mathrm{O}(2)$ fail
to act on $V_{n,2}$ in an orientation-preserving way, we could still use the
possible orientability of the quotients $V_{n,2}/H$ as an indication
of the extent to which $H$, 
as a whole, is compatible with the orientability 
of the several $V_{n,2}$. 
For example, while every element of $\mathrm{SO}(2)$ 
gives an orientation-preserving diffeomorphism on each $V_{n,2}$, 
it is well known that the Grassmannian $V_{n,2}/\mathrm{O}(2)$ 
of unoriented 2-planes in $\mathbb{R}^n$
is orientable if and only if $n$ is even (see for instance 
\cite[Example~47 on page~162]{prasolov}). 
We show that a similar---but {\it shifted}---result holds 
when $\mathrm{O}(2)$ is replaced by $D_8$.

\begin{notacion}\label{loscocientes}{\em
For a subgroup $H$ of $\mathrm{O}(2)$,
we will use the shorthand $V_{n,H}$ to denote the quotient 
$V_{n,2}/H$. For instance $V_{m+1,G}=E_{m,G}$, the space
in Notation~\ref{convenience}.
}\end{notacion}

\begin{proposicion}\label{orientabilidadB}
For $n>2$, $V_{n,D_8}$ is orientable if and only if $n$ is odd. 
Consequently, for $m>1$, the top dimensional cohomology group of 
$B(\P^m,2)$ is $$H^{2m-1}(B(\P^m,2))=\begin{cases}
\mathbb{Z}, & \mbox{{for even $\,m$}};\\\Z2, & 
\mbox{for odd $\,m$.}
\end{cases}$$
\end{proposicion}

\begin{nota}\label{analogousversionsF}{\em
Proposition~\ref{orientabilidadB} holds (with the same proof) if $D_8$ is 
replaced by its subgroup $\Z2\times\Z2$, and
$\bpm$ is replaced by $\fpm$.
It is interesting to compare both versions of
Proposition~\ref{orientabilidadB} with the fact that, for $m>1$, 
$\bpm$ is non-orientable, while $\fpm$ is orientable only for odd 
$m$~(\cite[Lemma~2.6]{SadokCohenFest}). 
}\end{nota}

\begin{ejemplo}\label{V22}{\em
The cases with $n=2$ and $m=1$ in Proposition~\ref{orientabilidadB} 
are special (compare 
to~\cite[Proposition~2.5]{SadokCohenFest}):
Since the 
quotient of $V_{2,2}=S^1\cup S^1$ by the action of 
$D_8\cap\mathrm{SO}(2)$ is diffeomorphic to 
the disjoint union of two copies of $S^1/\mathbb{Z}_4$, 
we see that $V_{2,D_8}\cong S^1$.
}\end{ejemplo}

If we take the same orientation for both circles
in $V_{2,2}=S^1\cup S^1$, it is clear that
the automorphism $H^1(V_{2,2})\to H^1(V_{2,2})$ 
induced by an element $r\in D_8$ 
is represented by the matrix
$\left(\begin{smallmatrix}0&1\\1&0\end{smallmatrix}\right)$
if $r\in \mathrm{SO}(2)$, but by the matrix 
$\left(\begin{smallmatrix}0&-1\\-1&0\end{smallmatrix}\right)$
if $r\not\in \mathrm{SO}(2)$. For larger values of $n$,
the method of proof of Proposition~\ref{orientabilidadB} 
allows us to describe the action of 
$D_8$ on the integral cohomology ring of $V_{n,2}$.
The answer is given in terms of the generators $\rho,\rho_1,\rho_2\in D_8$ 
introduced in Definition~\ref{inicio1Handel}.

\begin{teorema}\label{D8actionV}
The three automorphisms 
$\rho^*\!,\rho_1^*,\rho_2^*\colon H^q(V_{n,2})\to
H^q(V_{n,2})$ agree.
For $n>2$, this common morphism is the identity 
except when $n$ is even and $q\in\{n-2,2n-3\}$, 
in which case the common morphism is multiplication by $-1$.
\end{teorema}

Theorem~\ref{D8actionV} should be read keeping in mind the well-known 
cohomology ring $H^*(V_{n,2})$. We recall its simple description 
after proving Proposition~\ref{orientabilidadB}. For the time being it
suffices to recall, for the purposes of Proposition~\ref{SSOGimp} 
below, that $H^{n-1}(V_{n,2})=\Z2$ for odd $n$, $n\geq3$.

\medskip
We use our approach to Theorem~\ref{D8actionV}
in order to describe the integral cohomology ring of the 
oriented Grassmannian $V_{n,\mathrm{SO}(2)}$ for odd $n$, $n\geq3$.
Although the result might be well known ($V_{n,\mathrm{SO}(2)}$ is 
a complex quadric of complex dimension $n-2$), we include the details 
(an easy step from the constructions in this section) 
since we have not been able to find an explicit reference 
in the literature.

\begin{proposicion}\label{SSOGimp} 
Assume $n$ is odd, $n=2a+1$ with $a\geq1$. Let $\widetilde{z}\in 
H^2(V_{n,\mathrm{SO}(2)})$ stand for the Euler class of the 
smooth principal $S^1$-bundle
\begin{equation}\label{smoothfibrfix}
S^1\to V_{n,2}\to V_{n,\mathrm{SO}(2)}
\end{equation}
There is a class $\widetilde{x}\in 
H^{n-1}(V_{n,\mathrm{SO}(2)})$ mapping under the projection 
in~{\em(\ref{smoothfibrfix})} to the nontrivial element in $H^{n-1}(V_{n,2})$. 
Furthermore, as a ring, $H^*(V_{n,\mathrm{SO}(2)})=\mathbb{Z}
[\widetilde{x},\widetilde{z\hspace{.5mm}}]
\hspace{.5mm}/I_n$ where  $I_n$ is the ideal 
generated by
\begin{equation}\label{tresrelaciones}
\widetilde{x}^{\,2},\mbox{ \ }
{\widetilde{x}\,\widetilde{z}^{\,a}}, \mbox{ \ and \ \ }
\widetilde{z}^{\,a}-2\hspace{.4mm}\widetilde{x}.
\end{equation}
\end{proposicion}

It should be noted that the second generator of $I_n$ is superfluous. We 
include it in the description since it will become clear, from the proof of 
Proposition~\ref{SSOGimp}, that the first two terms in~(\ref{tresrelaciones})
correspond to the two families of differentials in the SSS of the fibration 
classifying~(\ref{smoothfibrfix}), 
while the last term corresponds to the family of 
nontrivial extensions in the resulting $E_\infty$-term.

\begin{nota}\label{compat1}{\em
It is illuminating to compare Proposition~\ref{SSOGimp} with
H.~F.~Lai's computation of the cohomology ring
$H^*(V_{n,\mathrm{SO}(2)})$ for even $n$, $n\geq4$. According 
to~\cite[Theorem~2]{lai}, $H^*(V_{2a,\mathrm{SO}(2)})=\mathbb{Z}
[\kappa,\widetilde{z\hspace{.5mm}}]\hspace{.5mm}/I_{2a}$ where 
$I_{2a}$ is the ideal generated by 
\begin{equation}\label{laisdescription}
\kappa^2-\varepsilon \kappa\widetilde{z}^{\,a-1}\mbox{ \ \ and \ \ }\,
\widetilde{z}^{\,a}-2\kappa\widetilde{z}.
\end{equation}
Here $\varepsilon=0$ for $a$ even, and $\varepsilon=1$ for $a$ odd, while
the generator $\kappa\in H^{2a-2}(V_{2a,\mathrm{SO}(2)})$ 
is the Poincar\'e dual
of the homology class represented by the canonical (realification) embedding 
$\mathbb{C}\P^{a-1}\hookrightarrow V_{2a,\mathrm{SO}(2)}$ (Lai also proves
that $(-1)^{a-1}\kappa \widetilde{z}^{\,a-1}$ is 
the top dimensional cohomology class in $V_{2a,\mathrm{SO}(2)}$
corresponding to the canonical orientation of this manifold). 
The first fact to observe in Lai's description of 
$H^*(V_{2a,\mathrm{SO}(2)})$
is that the two dimensionally forced relations $\kappa\widetilde{z}^{\,a}=0$ 
and $\widetilde{z}^{\,2a-1}=0$ can be algebraically 
deduced from the relations implied by~(\ref{laisdescription}). A similar 
situation holds for $H^*(V_{2a+1,\mathrm{SO}(2)})$, where 
the first two relations in~(\ref{tresrelaciones}), as well as the 
corresponding algebraically implied relation $\widetilde{z}^{\,2a}=0$, 
are forced by dimensional considerations. But it is more interesting
to compare Lai's result with Proposition~\ref{SSOGimp} 
through the canonical inclusions $\iota_n\colon
V_{n,\mathrm{SO}(2)}\hookrightarrow
V_{n+1,\mathrm{SO}(2)}$ ($n\geq3$). In fact, the 
relations given by the last element both in~(\ref{tresrelaciones}) 
and~(\ref{laisdescription}) readily give 
\begin{equation}\label{compatibilidadxk}
\iota_{2a}^*(\widetilde{x})=\kappa\widetilde{z}\mbox{ \ \ and \ \ }
\iota^*_{2a+1}(\kappa)=\widetilde{x}
\end{equation}
for $a\geq2$. Note that the second equality in~(\ref{compatibilidadxk})
can be proved, for all $a\geq1$, with the following alternative
argument: From~\cite[Theorem~2]{lai}, $2\kappa-\widetilde{z}^{\,a}\in
V_{2a+2,\mathrm{SO}(2)}$ is the Euler class of the canonical
{\it normal} bundle of $V_{2a+2,\mathrm{SO}(2)}$ and, therefore, maps 
trivially under $\iota_{2a+1}^*$. The second equality 
in~(\ref{compatibilidadxk}) then follows from
the relation implied by the last element in~(\ref{tresrelaciones}).
Needless to say, the usual cohomology ring $H^*(B\mathrm{SO}(2))$ 
is recovered as the inverse limit of the maps $\iota_n^*$
(of course $B\mathrm{SO}(2)\simeq\mathbb{C\P^\infty}$).
}\end{nota}

\begin{proof}[Proof of Proposition~{\em\ref{orientabilidadB}}
from Theorem~{\em\ref{D8actionV}}]
Since the action of every element in $D_8\cap\mathrm{SO}(2)$ preserves
orientation in $V_{n,2}$, and since two elements in 
$D_8-\mathrm{SO}(2)$ must
``differ'' by an orientation-preserving element in $D_8$,
the first assertion in Proposition~\ref{orientabilidadB} 
will follow once we argue that (say) $\rho$
is orientation-preserving precisely when $n$ is odd. But such a fact
is given by Theorem~\ref{D8actionV} in view of the UCT. 
The second assertion in Proposition~\ref{orientabilidadB} 
then follows from 
Lemma~\ref{inicio2Handel},~\cite[Corollary~3.28]{hatcher}, and the 
UCT (recall $\dim(V_{n,2})=2n-3$).
\end{proof}

We now start working toward the proof of Theorem~\ref{D8actionV}, 
recalling in particular the cohomology ring $H^*(V_{n,2})$.
Let $n>2$ and think of $V_{n,2}$ as the sphere bundle of the 
tangent bundle of $S^{n-1}$. The (integral cohomology) SSS for the fibration 
$S^{n-2}\stackrel{\iota}\to V_{n,2}\stackrel{\pi}\to S^{n-1}$ 
(where $\pi(v_1,v_2)=v_1$ and $\iota(w)=(e_1,(0,w))$ with 
$e_1=(1,0,\ldots,0)$) starts as
\begin{equation}\label{SSS}
{E}^{p,q}_2=\begin{cases}
\mathbb{Z}, & (p,q)\in\{(0,0),(n-1,0),(0,n-2),(n-1,n-2)\};
\\0, & \mbox{otherwise;}\end{cases}
\end{equation}
and the only possibly nonzero differential is multiplication by the Euler 
characteristic of $S^{{n-1}}$ (see for 
instance~\cite[pages 153--154]{mccleary}). 
At any rate, the only possibilities for a 
nonzero cohomology group $H^q(V_{n,2})$ are $\Z2$ or $\mathbb{Z}$. 
In the former case, 
any automorphism must be the identity. So the real task is to
determine the action of the three elements in Theorem~\ref{D8actionV} 
on a cohomology group $H^q(V_{n,2})=\mathbb{Z}$.

\begin{proof}[Proof of Theorem~{\em\ref{D8actionV}}]
The fact that $\rho^*=\rho_1^*=\rho_2^*$ follows by observing that 
the product of any two of the elements $\rho$, $\rho_1$, and $\rho_2$ lies 
in the path connected group $\mathrm{SO}(2)$, and therefore determines
an automorphism $V_{n,2}\to V_{n,2}$ which is homotopic to the identity.

\medskip
The analysis of the second assertion of Theorem~\ref{D8actionV} 
depends on the parity of $n$.

\medskip\noindent
{\bf Case with $n$ even, $n>2$.} 
The SSS~(\ref{SSS}) collapses, giving that $H^*(V_{n,2})$
is an exterior algebra (over $\mathbb{Z}$) 
on a pair of generators $x_{n-2}$ and $x_{n-1}$ (indices
denote dimensions). The spectral sequence also gives that $x_{n-2}$ 
maps under $\iota^*$ 
to the generator in $S^{n-2}$, whereas $x_{n-1}$ is the image under $\pi^*$
of the generator in $S^{n-1}$. Now, the (obviously) commutative diagram

\begin{picture}(0,75)(-103,0)
\put(0,60){$S^{n-2}$}
\put(8,55){\vector(0,-1){15}}
\put(84,55){\vector(0,-1){15}}
\put(3,47){\scriptsize $\iota$}
\put(87,47){\scriptsize $\iota$}
\put(20,63){\vector(1,0){54}}
\put(27,66){\scriptsize antipodal map}
\put(43.5,36){\scriptsize $\rho_2$}
\put(20,33){\vector(1,0){54}}
\put(76,60){$S^{n-2}$}
\put(79,30){$V_{n,2}$}
\put(3,30){$V_{n,2}$}
\put(38,5){$S^{n-1}$}
\put(15,24){\vector(2,-1){22}}
\put(21,13.5){\scriptsize $\pi$}
\put(78,25){\vector(-2,-1){22}}
\put(68,14){\scriptsize $\pi$}
\end{picture}

\noindent 
implies that $\rho_2^*$ (and therefore $\rho_1^*$ and $\rho^*$) is the 
identity on $H^{n-1}(V_{n,2})$, and that $\rho_2^*$ (and therefore 
$\rho_1^*$ and $\rho^*$) act by multiplication by $-1$ on 
$H^{n-2}(V_{n,2})$.
The multiplicative structure then implies that the last assertion holds
also {on} $H^{2n-3}(V_{n,2})$. 

\medskip\noindent
{\bf Case with $n$ odd, $n>2$.} 
The description in~(\ref{SSS}) 
of the start of the SSS implies that the only nonzero
cohomology groups of $V_{n,2}$ are $H^{n-1}(V_{n,2})=\Z2$ and 
$H^i(V_{n,2})=\mathbb{Z}$ for $i=0,2n-3$. Thus,
we only need to make sure that 
\begin{equation}\label{mostrar}
\mbox{$\rho^*\colon H^{2n-3}(V_{n,2})
\to H^{2n-3}(V_{n,2})$ is the identity morphism.} 
\end{equation}
Choose generators
$x\in H^{n-1}(V_{n,2})$, $y\in H^{2n-3}(V_{n,2})$, and 
$z\in H^2(\mathbb{C}\P^\infty)$, and let 
$V_{n,\mathrm{SO}(2)}\to\mathbb{C}\P^\infty$ classify the circle 
fibration~(\ref{smoothfibrfix}).
Thus, the $E_2$-term of the SSS for the fibration
\begin{equation}\label{fibrationOG}
V_{n,2}\to V_{n,\mathrm{SO}(2)}\to\mathbb{C}\P^\infty
\end{equation}
takes the simple form

\begin{picture}(0,94)(-2,-14)
\put(0,0){\line(1,0){270}}
\put(0,0){\line(0,1){74}}
\multiput(-2,-2)(20,0){4}{$\mathbb{Z}$}
\multiput(118,-2)(20,0){3}{$\mathbb{Z}$}
\multiput(218,-2)(20,0){3}{$\mathbb{Z}$}
\multiput(-2,29)(20,0){4}{$\bullet$}
\multiput(118,29)(20,0){3}{$\bullet$}
\multiput(218,29)(20,0){3}{$\bullet$}
\multiput(-2,58)(20,0){4}{$\mathbb{Z}$}
\multiput(118,58)(20,0){3}{$\mathbb{Z}$}
\multiput(218,58)(20,0){3}{$\mathbb{Z}$}
\put(-1,-8){\scriptsize$1$}
\put(19,-8){\scriptsize$z$}
\put(39,-8){\scriptsize$z^2$}
\put(59,-8){\scriptsize$z^3$}
\put(117,-8){\scriptsize$z^{a-1}$}
\put(138,-8){\scriptsize$z^{a}$}
\put(157,-8){\scriptsize$z^{a+1}$}
\put(217,-8){\scriptsize$z^{n-2}$}
\put(237,-8){\scriptsize$z^{n-1}$}
\put(258,-8){\scriptsize$z^{n}$}
\put(272,-7){$\dots$}
\put(272,30){$\dots$}
\put(272,60){$\dots$}
\put(187,-7){$\dots$}
\put(187,30){$\dots$}
\put(187,60){$\dots$}
\put(87,-7){$\dots$}
\put(87,30){$\dots$}
\put(87,60){$\dots$}
\put(-9,29.5){\scriptsize$x$}
\put(-9,58){\scriptsize$y$}
\end{picture}

\noindent where $n=2a+1$, and a bullet represents a copy of $\Z2$.
The proof of Proposition~\ref{SSOGimp} below gives two rounds
of differentials, both originating on the top horizontal line;
the element $2y$ is a cycle in the first round of differentials, 
but determines the second round of differentials by 
\begin{equation}\label{differential}
d_{2n-2}(2y)=z^{n-1}.
\end{equation}
The key ingredient 
comes from the observation that $\rho$ and the
involution $\tau\colon V_{n,\mathrm{SO}(2)}
\to V_{n,\mathrm{SO}(2)}$ 
that reverses orientation of an oriented $2$-plane
fit into the pull-back diagram

\begin{equation}\label{conjugation}
\begin{picture}(0,36)(50,32)
\put(0,60){$V_{n,2}$}
\put(8,55){\vector(0,-1){14}}
\put(84,55){\vector(0,-1){14}}
\put(20,63){\vector(1,0){52}}
\put(45,66){\scriptsize $\rho$}
\put(45,36){\scriptsize $\tau$}
\put(45,6){\scriptsize $c$}
\put(24,33){\vector(1,0){46}}
\put(76,60){$V_{n,2}$}
\put(74,30){$V_{n,\mathrm{SO}(2)}$}
\put(-5,30){$V_{n,\mathrm{SO}(2)}$}
\put(2,0){$\mathbb{C}\P^\infty$}
\put(78,0){$\mathbb{C}\P^\infty$}
\put(24,3){\vector(1,0){50}}
\put(8,25){\vector(0,-1){14}}
\put(84,25){\vector(0,-1){14}}
\end{picture}
\end{equation}

\vspace{1.7cm}\noindent 
where $c$ stands for conjugation. [Indeed, thinking of 
$V_{n,\mathrm{SO}(2)}
\to \mathbb{C}\P^\infty$ as an inclusion, $\tau$ is the restriction of $c$, 
and $\rho$ becomes the equivalence induced on (selected) fibers.]
Of course $c^*(z)=-z$ in $H^2(\mathbb{C}\P^\infty)$, so that
\begin{equation}\label{par}c^*(z^{n-1})=z^{n-1}\end{equation}
(recall $n$ is odd). Thus, in terms of the map of spectral sequences
determined by~(\ref{conjugation}), 
conditions~(\ref{differential}) and~(\ref{par}) force the relation
$\rho^*(2y)=2y$. This gives~(\ref{mostrar}).
\end{proof}

The proof of~(\ref{mostrar}) we just gave (for odd $n$)
can be simplified by working over the rationals 
(see Remark~\ref{transfer} in the next paragraph). We have chosen the 
spectral sequence analysis of~(\ref{fibrationOG}) since it
leads us to Proposition~\ref{SSOGimp}.

\begin{nota}\label{transfer}{\em
It is well known that whenever a finite group $H$ acts freely 
on a space $X$, with $Y = X/H$, the rational cohomology of $Y$ maps 
isomorphically onto the $H$-invariant elements in the rational 
cohomology of $X$ (see for instance~\cite[Proposition~3G.1]{hatcher}). 
We apply this fact to the $8$-fold covering projection 
$\,\theta\colon V_{n,2}\to V_{n,D_8}$. Since the only nontrivial groups 
$H^q(V_{n,2};\mathbb{Q})$ are $\mathbb{Q}$ for $q=0,2n-3$ (this is where 
we use that $n$ is odd), we get that 
the rational cohomology of $V_{n,D_8}$ is $\mathbb{Q}$ in dimension $0$, 
vanishes in positive dimensions below $2n-3$, and is either $\mathbb{Q}$ 
or $0$ in the top dimension $2n-3$.  
But $V_{n,D_8}$ is a manifold of odd dimension, so its Euler 
characteristic is zero; this forces the top rational cohomology 
to be $\mathbb{Q}$. Thus, every element in $D_8$ acts as the identity 
on the top rational (and therefore integral) cohomology group
of $V_{n,2}$. This gives in particular~(\ref{mostrar}), the 
real content of Theorem~\ref{D8actionV} for an odd $n$.}
\end{nota}

As in the notation introduced right after~(\ref{mostrar}),
let $z\in H^2(\mathbb{C}\P^\infty)$ be a generator so that
the element $\widetilde{z}\in H^2(V_{n,\mathrm{SO}(2)})$ in 
Proposition~\ref{SSOGimp} is the image of $z$ under the projection 
map in~(\ref{fibrationOG}).

\begin{proof}[Proof of Proposition~{\em\ref{SSOGimp}}]
The $E_2$-term of the SSS for~(\ref{fibrationOG}) has been indicated 
in the proof of Theorem~\ref{D8actionV}. In that picture, 
the horizontal $x$-line consists of permanent cycles; indeed,
there is no nontrivial target in a $\mathbb{Z}$ group
for a differential originating at a $\Z2$ group.
Since $\dim(V_{n,\mathrm{SO}(2)})=2n-4$, the term $xz^a$ must be killed 
by a differential, and the only way this can happen is by means of
$d_{n-1}(y)=xz^a$. By multiplicativity, this settles a whole family of 
differentials killing off the elements $xz^i$ with $i\geq a$. Note that 
this still leaves groups $2\cdot\mathbb{Z}$ in the $y$-line (rather,
the $2y$-line). Just as before, dimensionality forces the 
differential~(\ref{differential}), and multiplicativity determines 
a corresponding family of differentials. What remains in the SSS after 
these two rounds of differentials---depicted below---consists of
permanent cycles, so the spectral sequence collapses from this point on.

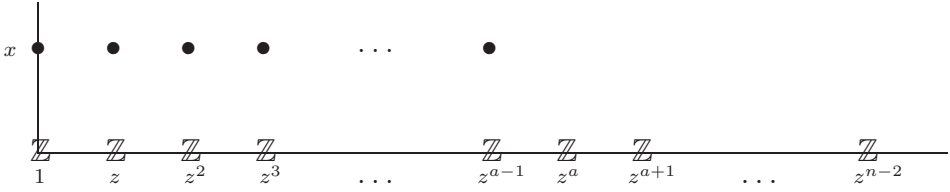
\begin{picture}(0,53)(-24,-5)
\put(0,0){\line(1,0){242}}
\put(0,0){\line(0,1){40}}
\multiput(-2,-2)(20,0){4}{$\mathbb{Z}$}
\put(85,-8){$\ldots$}
\put(85,27){$\ldots$}
\put(187,-8){$\ldots$}
\multiput(118,-2)(20,0){3}{$\mathbb{Z}$}
\multiput(218,-2)(20,0){1}{$\mathbb{Z}$}
\multiput(-2,26)(20,0){4}{$\bullet$}
\multiput(118,26)(20,0){1}{$\bullet$}
\put(-1,-8){\scriptsize$1$}
\put(19,-8){\scriptsize$z$}
\put(39,-8){\scriptsize$z^2$}
\put(59,-8){\scriptsize$z^3$}
\put(117,-8){\scriptsize$z^{a-1}$}
\put(138,-8){\scriptsize$z^{a}$}
\put(157,-8){\scriptsize$z^{a+1}$}
\put(217,-8){\scriptsize$z^{n-2}$}
\put(-9,26){\scriptsize$x$}
\end{picture}

\bigskip\noindent Finally, we note that all possible extensions 
are nontrivial. Indeed, orientability of 
$\hspace{.5mm}V_{n,\mathrm{SO}(2)}$ gives  
$H^{2n-4}(V_{n,\mathrm{SO}(2)})=\mathbb{Z}$, which implies 
a nontrivial extension involving $xz^{a-1}$ and $z^{n-2}$. Since 
multiplication by $z$ is monic in total dimensions
less that $2n-4$ of the $E_\infty$-term, 
the $5$-Lemma (applied recursively) shows that
the same assertion is true in $H^{*}(V_{n,\mathrm{SO}(2)})$.
This forces the corresponding nontrivial 
extensions in degrees lower than 
$2n-4$: an element of order 2 in low dimensions would produce,
after multiplication by $z$, a corresponding element of order $2$ in the 
top dimension. The proposition follows.
\end{proof}

Lai's description of the ring $H^*(V_{2a,\mathrm{SO}(2)})$ 
given in Remark~\ref{compat1} can be used to understand 
the full patter of differentials and extensions in 
the SSS of~(\ref{fibrationOG}) for 
$n=2a\hspace{.5mm}$. Due to space limitations, details are not given 
here---but they are discussed in Remark~3.10 of the 
preliminary version~\cite{v1} of this paper.

\medskip
We close this section with an argument 
that explains, in a geometric way, the switch in parity of $n$ 
when comparing the orientability properties of $V_{n,\mathrm{O}(2)}$ 
to those of $V_{n,D_8}$. Let $\pi$ stand for the projection map in
the smooth fiber bundle~(\ref{smoothfibrfix}). The tangent bundle 
$T_{n,2}$ to $V_{n.2}$ decomposes as the Whitney sum $$T_{n,2}\cong\pi^*
(T_{n,\mathrm{SO}(2)})\oplus\lambda$$ where $T_{n,\mathrm{SO}(2)}$
is the tangent bundle to $V_{n,\mathrm{SO}(2)}$, and $\lambda$ 
is the $1$-dimensional bundle of tangents to the fibers---a trivial bundle 
since we have the nowhere vanishing vector field obtained by differentiating
the free action of $S^1$ on $V_{n.2}$. Note that $\rho\colon V_{n,2}\to V_{n,2}$ 
reverses orientation on all fibers and so reverses a given orientation 
of $\lambda$. Hence, $\rho\,$ {\it preserves\hspace{.5mm}} 
a chosen orientation of $T_{n,2}$
precisely when the involution $\tau$ in~(\ref{conjugation})
{\it reverses} a chosen orientation of $T_{n,\mathrm{SO}(2)}$. But,
as explained in the proof of Proposition~\ref{orientabilidadB}, 
$V_{n,D_8}$ is orientable precisely when 
$\rho$ is orientation-preserving. Likewise, $V_{n,\mathrm{O}(2)}$
is orientable precisely when $\tau$ is orientation-preserving.

\section{Torsion linking form and Theorems~\ref{descripcionordenada} 
and~\ref{descripciondesaordenada}}\label{linkinsection}
In this short section we outline an argument, based on the
classical torsion linking form, that allows us to compute the cohomology 
groups described by Theorems~\ref{descripcionordenada} 
and~\ref{descripciondesaordenada} in all but three critical dimensions.
The totality of dimensions (together with the proofs of 
Proposition~\ref{mono4} and Theorems~\ref{HFpar} and~\ref{HF1m4})
is considered in the next three sections---the first
two of which represent, together with the final 
Section~\ref{STCP67},  the bulk of spectral sequence 
computations in this paper.

\medskip
For a space $X$ let $TH_i(X;A)$ (respectively, $TH^i(X;A)$) denote 
the torsion subgroup of 
the $i^{\mathrm{th}}$ homology (respectively, cohomology) 
group of $X$ with (possibly twisted) 
coefficients $A$. As usual, omission of $A$ from the notation 
indicates that a simple system of 
$\mathbb{Z}$-coefficients is used. We are interested in the twisted 
coefficients $\widetilde{\mathbb{Z}}$ arising from the orientation character
of a closed $m$-manifold $X=M$ for, in such a case, there are non-singular
pairings 
\begin{equation}\label{linkingtorsion}
TH^i(M)\times TH^j(M;\widetilde{\mathbb{Z}})\to\mathbb{Q}/
\mathbb{Z}
\end{equation}
(for $i+j=m+1$), 
the so-called torsion linking forms, constructed from the UCT 
and Poincar\'e duality. Although~(\ref{linkingtorsion}) 
seems to be best known for an orientable $M$ (see for
instance~\cite[pages 16--17 
and 58--59]{collins}), the construction works just as well in a 
non-orientable setting.
We briefly recall the details (in cohomological terms) for completeness.

\medskip
Start by observing that for a finitely generated abelian group 
$H=F \oplus T$ with $F$ free abelian and $T$ a finite group, the group
$\mathrm{Ext}^1(H,\mathbb{Z})\cong\mathrm{Ext}^1(T,\mathbb{Z})$ 
is canonically isomorphic
to $\mathrm{Hom}(T,\mathbb{Q}/\mathbb{Z})$, the Pontryagin dual of $T$
(verify this by using the exact 
sequence $0 \to \mathbb{Z} \to \mathbb{Q} \to \mathbb{Q}/\mathbb{Z} \to 0$, 
and noting that $\mathbb{Q}$ is injective while $\mathrm{Hom}(T,\mathbb{Q})
=0$). In particular, the canonical isomorphism $TH^i(M)\cong\mathrm{Ext}^1
(TH_{i-1}(M),\mathbb{Z})$ coming from the UCT yields a non-singular 
pairing $TH^i(M)\times TH_{i-1}(M)\to \mathbb{Q}/\mathbb{Z}$. 
The form in~(\ref{linkingtorsion}) then follows by using 
Poincar\'e duality (in its not necessarily orientable version, 
see~\cite[Theorem~3H.6]{hatcher} or~\cite[Theorem~4.51]{ranicki}).
As explained by Barden in \cite[Section~0.7]{barden} (in the orientable case),
the resulting pairing can be interpreted geometrically 
as the classical torsion linking number~(\cite{KM,ST,wall}). 

\medskip
Recall the group $G$ and orbit space $E_m$ in 
Notation~\ref{convenience}.
We next indicate how the isomorphisms
\begin{equation}\label{linkisos}
TH^i(M)\cong TH^j(M;\widetilde{\mathbb{Z}}),\quad i+j=2m,
\end{equation}
coming from~(\ref{linkingtorsion}) for $M=E_m$ 
can be used for computing 
most of the integral cohomology groups of $F(\P^m,2)$ 
and $B(\P^m,2)$.

\medskip
Since $V_{m+1,2}$ is 
($m-2$)-connected\footnote{Low 
dimensional cases with $m\leq3$ are given special attention 
in Example~\ref{babyexample}, Remark~\ref{elm1}, and~(\ref{casoespecial}) 
in the following sections.},
the map in~(\ref{lap}) 
is ($m-1$)-connected. Therefore it induces an isomorphism
(respectively, monomorphism) in cohomology with any---possibly 
twisted, in view of~\cite[Theorem~$6.4.3^*$]{whitehead}---coefficients 
in dimensions $i\leq m-2$ (respectively, $i=m-1$). Together 
with Corollary~\ref{HBD} and Lemmas~\ref{inicio2Handel} and~\ref{kunneth}, 
this leads to the explicit 
description of the groups in Theorems~\ref{descripcionordenada} 
and~\ref{descripciondesaordenada}\hspace{.5mm} 
in dimensions at most $m-2$. The corresponding  
groups in dimensions at least $m+2$ can then be obtained from
the isomorphisms~(\ref{linkisos}) 
and the full description in Section~\ref{HBprelim}
of the twisted and untwisted cohomology groups of $BG$.
Note that the last step requires knowing that, when  
$E_m$ is non-orientable (as determined in 
Proposition~\ref{orientabilidadB} and Remark~\ref{analogousversionsF}), the 
twisted coefficients $\widetilde{\mathbb{Z}}$ agree with those 
$\mathbb{Z}_\alpha$ used in Theorem~\ref{modulestructure}. But such a 
requirement is a direct consequence of Theorem~\ref{D8actionV}.  
Since the torsion-free subgroups of $H^*(E_m)$ are easily identifiable from 
a quick glance at the $E_2$-term of the 
CLSS for the $G$-action on $V_{m+1,2}$,
only the torsion subgroups in 
Theorems~\ref{descripcionordenada} and~\ref{descripciondesaordenada} 
in dimensions 
\begin{equation}\label{faltantes} 
\mbox{$m-1$, $\;m$, $\;$and $\,\;m+1$}
\end{equation} 
are lacking description in this argument.

\medskip
A deeper analysis of the CLSS of the $G$-action on $V_{m+1,2}$
(worked out in Sections~\ref{HB} and~\ref{problems3} for $G=D_8$, 
and discussed briefly in Section~\ref{HF} for $G=\Z2\times\Z2$)
will give us (among other things) a detailed description of 
the three missing cases in~(\ref{faltantes}) {\it except} 
for the ($m+1$)-dimensional group when $G=D_8$ 
and $m\equiv3\bmod4$. Note that this apparently
singular case cannot be handled directly with the torsion linking form 
argument in the previous paragraph because the connectivity 
of $V_{m+1,2}$ only gives
the injectivity, but not the surjectivity, of the first map in the composite
\begin{equation}\label{excepcional}
H^{m-1}(BD_8;\mathbb{Z}_\alpha)\stackrel{p^*}\longrightarrow 
H^{m-1}(\bpm;\mathbb{Z}_\alpha)\cong H^{m+1}(\bpm).
\end{equation}
To overcome the problem, in Section~\ref{problems3} we perform 
a direct calculation in the first two pages of the Bockstein spectral 
sequence (BSS) of $B(\P^{4a+3},2)$ to prove that~(\ref{excepcional}) is 
indeed an isomorphism for $m\equiv3\bmod4$---therefore 
completing the proof of Theorems~\ref{descripcionordenada} 
and~\ref{descripciondesaordenada}.

\medskip
\begin{table}[h]
\centerline{
\begin{tabular}{|c|c|c|c|c|c|c|c|c|c|c|c|c|c|c|c|}
\hline
\rule{0pt}{15pt}$*={}$& 2 & 3 & 4 & 5 & 6 & 7 & 8 & 9 & 10 & 
11 & 12 & 13 & 14\\[0.5ex]
\hline
\rule{0pt}{15pt}$H^*(E_{2,D_8})$ &$\langle2\rangle$&&&&&&&&&&&&\\ [0.5ex]
\hline
\rule{0pt}{15pt}$H^*(E_{4,D_8})$ &$\langle2\rangle$&$\langle1\rangle$
&\{2\}&
$\langle1\rangle$&$\langle2\rangle$&&&&&&&&\\ [0.5ex]
\hline
\rule{0pt}{15pt}$H^*(E_{6,D_8})$ &$\langle2\rangle$&$\langle1\rangle$&
\{2\}
&$\langle2\rangle$&$\langle4\rangle$&$\langle2\rangle$&\{2\}&
$\langle1\rangle$&$\langle2\rangle$&&&&\\ [0.5ex]
\hline
\rule{0pt}{15pt}$H^*(E_{8,D_8})$ &$\langle2\rangle$&$\langle1\rangle$&
\{2\}&$\langle2\rangle$&$\langle4\rangle$&
$\langle3\rangle$&\{4\}
&$\langle3\rangle$&$\langle4\rangle$&$\langle2\rangle$
&\{2\}&$\langle1\rangle$&$\langle2\rangle$\\ [0.5ex]
\hline
\end{tabular}}
\caption{$H^*(E_{m,D_8})\cong H^*(B(\P^m,2))\,$ for 
$m=2,4,6$, and $8$
\label{tabla}}
\end{table}

The isomorphisms in~(\ref{linkisos}) yield a 
(twisted, in the non-orientable case) 
symmetry for the torsion groups of $H^*(E_m)$. This is 
illustrated (for $G=D_8$ and in the orientable case) 
in Table~\ref{tabla} following the conventions set 
in the very first paragraph of the paper.

\section{Case of $B(\P^m,2)$ for 
$m\not\equiv3\bmod4$}\label{HB}

This section and the next one contain
a careful study of the CLSS of the $D_8$-action on
$V_{m+1,2}$ described in Definition~\ref{inicio1Handel}; 
the corresponding (much simpler) analysis for the restricted 
$(\Z2\times\Z2)$-action is outlined in Section~\ref{HF}. The CLSS 
approach will yield, in addition, direct proofs of 
Proposition~\ref{mono4} and Theorems~\ref{HFpar} and~\ref{HF1m4}.
The reader is assumed to be familiar with the properties of the CLSS
of a regular covering space, complete details of which first appeared 
in~\cite{cartan}.

\medskip
We start with the less involved situation of an even $m$ and,
as a warm-up, we consider first the case $m=2$.

\begin{ejemplo}\label{babyexample}{\em
Lemmas~\ref{inicio2Handel} and~\ref{handel68D}, Corollary~\ref{HBD}, and 
Theorem~\ref{D8actionV}
imply that, in total dimensions at most $\dim(V_{3,D_8})=3$,
the (integral cohomology) CLSS for the $D_8$-action on $V_{3,2}$
starts as 

\begin{picture}(0,88)(-100,-14)
\qbezier[50](0,0)(45,0)(90,0) 
\qbezier[40](0,0)(0,33)(0,66) 
\multiput(-2.5,-3)(25,0){1}{$\mathbb{Z}$}
\put(22.5,-11){\scriptsize$1$}
\put(47.5,-11){\scriptsize$2$}
\put(72.5,-11){\scriptsize$3$}
\put(-10,17.5){\scriptsize$1$}
\put(-10,37.5){\scriptsize$2$}
\put(-10,57.5){\scriptsize$3$}
\put(43.5,-3){$\langle2\rangle$}
\put(68.5,-3){$\langle1\rangle$}
\put(-5.35,37.1){$\langle1\rangle$}
\put(19,37.1){$\langle2\rangle$}
\put(-3,56.8){$\mathbb{Z}$}
\end{picture}

\noindent The only possible nontrivial differential in this range 
is $d_3^{\,0,2}\colon E_2^{\,0,2}\to E_2^{\,3,0}$, which must be an
isomorphism in view of the second assertion in 
Proposition~\ref{orientabilidadB}. This yields the $\P^2$-case in 
Proposition~\ref{mono4} and Theorems~\ref{descripciondesaordenada} 
and~\ref{HFpar} (with $G=D_8$ in the latter one).
As indicated in Table~\ref{tabla}, the symmetry
isomorphisms are invisible in the current situation. 
It is worth noticing that the $d_3$-differential
originating at node $(1,2)$ must be injective. This observation 
will be the basis in our argument for the general situation, where 
2-rank considerations will be the catalyst. Here and in what follows,
by the {\it {\rm 2}-rank} (or simply {\it rank})
of a finite abelian 2-group $H$ we mean the rank 
($\F2$-dimension) of $H\otimes\F2$.
}\end{ejemplo}

\begin{proof}[Proof of Theorem~{\em\ref{HFpar}} 
for $G\hspace{.8mm}{=}\hspace{.8mm}D_8$, 
and of Proposition~{\em\ref{mono4}}, 
both with $\,m$ even, $m\hspace{.8mm}{\geq}\hspace{.8mm}4$]
The assertion in Theorem~\ref{HFpar} for 

\vspace{-3mm}
\begin{itemize}
\item $i\geq2m$ follows from 
Lemma~\ref{inicio2Handel}
and the fact that $\dim(V_{m+1,2})=2m-1$, and for

\vspace{-2mm}
\item $i=2m-1$ follows from the fact that $H^{2m-1}(BD_8)$
is a torsion group (Corollary~\ref{HBD}) while $H^{2m-1}(\bpm)=\mathbb{Z}$ 
(Proposition~\ref{orientabilidadB}).
\end{itemize} 

\vspace{-3mm}\noindent
We work with the (integral cohomology) 
CLSS for the $D_8$-action on $V_{m+1,2}$
in order to prove Propostion~\ref{mono4} 
and the  assertions in Theorem~\ref{HFpar} for $i<2m-1$.

\smallskip
In view of Theorem~\ref{D8actionV}, the spectral sequence
has a simple system of coefficients and, from the 
description of $H^*(V_{m+1,2})$ in the proof of Theorem~\ref{D8actionV},
it is concentrated in the three horizontal lines with $q=0,m,2m-1$. We can
focus on the lines with $q=0,m$ in view of the range under 
current consideration. At the start of the CLSS
there is a copy of
 
\vspace{-3mm}
\begin{itemize}
\item $H^*(BD_8)$ (described by 
Corollary~\ref{HBD}) at the line with $q=0$;

\vspace{-2mm}
\item $H^*(BD_8,\F2)$
(described by Lemma~\ref{handel68D}) at the line with $q=m$.
\end{itemize} 

\vspace{-3mm}\noindent
Note that the assertion in Theorem~\ref{HFpar}
for $i<m$ is an obvious consequence of the 
above description of the $E_2$-term of the CLSS. The case $i=m$ 
will follow once we show that the ``first'' potentially nontrivial
differential $d_{m+1}^{\hspace{.25mm}0,m}
\colon E_2^{0,m}\to E_2^{m+1,0}$ is injective.
More generally, we show in the paragraph 
following~(\ref{sharp}) below that 
\begin{equation}\label{difiny}
\mbox{all differentials $d_{m+1}^{\hspace{.25mm}m-\ell-1,m}
\colon E_2^{m-\ell-1,m}\to E_2^{2m-\ell,0}$ 
with $0<\ell<m$ are injective.}
\end{equation}
From this, the assertion in Theorem~\ref{HFpar} 
for $m<i<2m-1$ follows at once.

\smallskip
The information we need about differentials is 
forced by the ``size'' of their domains and codomains.
For instance, since $H^{2m-1}(\bpm)$ is torsion-free, all of 
$E_2^{2m-1,0}=H^{2m-1}(BD_8)=\langle m-1\rangle$ must be killed 
by differentials. But the only possibly nontrivial
differential landing in 
$E_2^{2m-1,0}$ is the one in~(\ref{difiny}) with $\ell=1$. 
The resulting surjective $d_{m+1}^{m-2,m}$ map 
must be an isomorphism since its domain, $E_2^{m-2,m}=H^{m-2}(BD_8;{\F2})=
\langle m-1\rangle$, is isomorphic to its codomain.

\medskip
The extra input we need in order 
to deal with the rest of the differentials in~(\ref{difiny})
comes from the short exact sequences
\begin{equation}\label{spliteadas}
0\to\mbox{Coker}(2_{i})\to H^i(\bpm;{\F2})\to\mbox{Ker}(2_{i+1})\to0
\end{equation}
obtained from the Bockstein long exact sequence 
$$ 
\cdots\leftarrow H^i(B(\P^m,2);{\F2})\stackrel{\pi_i}\leftarrow
H^i(B(\P^m,2))\stackrel{2_i}\leftarrow 
H^i(B(\P^m,2))\stackrel{\partial_i}\leftarrow H^{i-1}(B(\P^m,2);{\F2})
\leftarrow\cdots.
$$ 
From the $E_2$-term of the spectral sequence we easily see 
that ($H^1(\bpm)=0$ and that) $H^i(\bpm)$ {is} a finite 
$2$-torsion group for $1<i<2m-1$; let $r_i$ denote its $2$-rank. Then 
Ker$(2_i)\cong{}$Coker$(2_i)\cong\langle r_i\rangle$, so 
that~(\ref{spliteadas}),
Corollary~\ref{handel68B}, and an easy induction
(grounded by the fact that $\mbox{Ker}(2_{2m-1})=0$, in view 
of the second assertion in 
Proposition~\ref{orientabilidadB}) yield
\begin{equation}\label{ranks}
r_{2m-\ell}=\begin{cases}a+1,&\ell=2a;\\a,&\ell=2a+1;\end{cases}
\end{equation}
for $2\leq\ell\leq m-1$. Under these conditions, the $\ell$-th differential
in~(\ref{difiny}) takes the form
\begin{equation}\label{sharp}
\langle m-\ell\rangle\hspace{.5mm}{=}\hspace{.5mm}H^{m-\ell-1}(BD_8;{\F2})
\to H^{2m-\ell}(BD_8)\,{=}\begin{cases}
\left\{m-\frac\ell2\right\}
,&\ell\equiv0\bmod4;\\
\left\langle m-\frac{\ell-2}2\right\rangle,&\ell\equiv2\bmod4;\\
\left\langle m-\frac{\ell+1}2\right\rangle,&\mbox{otherwise}.\\
\end{cases}
\end{equation}
But the cokernel of this map, which is a subgroup of $H^{2m-\ell}(\bpm)$,
must have 2-rank at most $r_{2m-\ell}$. An easy counting argument
(using the right exactness of the tensor product) shows that this is 
possible only with an injective differential~(\ref{sharp}) 
which, in the case of $\ell\equiv0\bmod4$, yields an injective 
map even after tensoring\footnote{This amounts to the fact that 
twice the generator of the $\mathbb{Z}_4$-summand in~(\ref{sharp}) 
is not in the image of~(\ref{sharp})---compare to the proof of 
Proposition~\ref{algebraico}.} with $\Z2$.

\medskip
Note that, in total dimensions at most $2m-2$, 
the $E_{m+2}$-term of the spectral sequence is concentrated 
on the base line ($q=0$). Thus, for $2\leq\ell\leq m-1$, 
$H^{2m-\ell}(\bpm)$ is the cokernel of the 
differential~(\ref{sharp})---which yields the surjectivity 
asserted in Theorem~\ref{HFpar} in the range $m<i<2m-1$. 
Furthermore the kernel of $p^*\colon 
H^{2m-\ell}(BD_8)\to H^{2m-\ell}(\bpm)$ is the elementary abelian 2-group 
specified on the left hand side of~(\ref{sharp}). In fact, the observation in
the second half of the final assertion in the previous paragraph proves 
Proposition~\ref{mono4}.
\end{proof}

As indicated in the last paragraph of the previous proof, for
$2\leq\ell\leq m-1$ the CLSS 
analysis identifies the group $H^{2m-\ell}(\bpm)$
as the cokernel of~(\ref{sharp}). Thus, the following algebraic 
calculation of these groups not only gives us
an alternative approach to that 
using the non-singularity of the torsion linking form, but it also allows 
us to recover (for $m$ even and $G=D_8$) the three missing cases 
in~(\ref{faltantes})---therefore completing the proof of 
the $\P^{\mathrm{even}}$-case of Theorem~\ref{descripciondesaordenada}.

\begin{proposicion}\label{algebraico}
For $2\leq\ell\leq m-1$, 
the cokernel of the differential~{\em(\ref{sharp})} is isomorphic to
$$
H^{2m-\ell}(\bpm)=\begin{cases}
\left\{\frac\ell2\right\},&\ell\equiv0\bmod4;\\
\left\langle\frac{\ell}2+1\right\rangle,&\ell\equiv2\bmod4;\\
\left\langle\frac{\ell-1}2\right\rangle,&\mbox{otherwise}.
\end{cases}
$$
\end{proposicion}
\begin{proof}
Cases with $\ell\not\equiv0\bmod4$ follow from a simple count, 
so we only offer an argument for $\ell\equiv0\bmod4$. Consider the diagram
with exact rows

\begin{picture}(0,65)(-36.5,-6)
\put(0,40){$0$}
\put(8,42){\vector(1,0){15}}
\put(26,40){$\langle m-\ell\rangle$}
\put(55,42){\vector(1,0){22}}
\put(80,40){$\left\{ m-\frac{\ell}2\right\}$}
\put(115,42){\vector(1,0){22}}
\put(140,40){$H^{2m-\ell}(\bpm)$}
\put(205,42){\vector(1,0){15}}
\put(224,40){$0$}
\put(0,3){$0$}
\put(8,5){\vector(1,0){15}}
\put(26,3){$\langle m-\ell\rangle$}
\put(55,5){\vector(1,0){17}}
\put(75,3){$\left\langle m-\frac{\ell}2+1\right\rangle$}
\put(122,5){\vector(1,0){33}}
\put(158,3){$\left\langle\frac{\ell}2+1\right\rangle$}
\put(185,5){\vector(1,0){35}}
\put(224,3){$0$}
\put(39,14){\line(0,1){20}}
\put(40.5,14){\line(0,1){20}}
\put(95,14){\vector(0,1){20}}
\put(93.5,14){\oval(3,3)[b]}
\put(172,14){\vector(0,1){20}}
\put(170.5,14){\oval(3,3)[b]}
\end{picture}

\noindent 
where the top horizontal monomorphism is~(\ref{sharp}), and where the middle
group on the bottom is included in the top one as the elements annihilated by
multiplication by~$2$. The lower right group is $\langle\frac\ell2+1\rangle$
by a simple counting. The snake lemma shows that the right-hand-side vertical 
map is injective with cokernel $\Z2$; the resulting extension is 
nontrivial in view of~(\ref{ranks}).
\end{proof}

\begin{ejemplo}\label{muchkernel1}{\em 
For $m$ even,~\cite[Theorem~1.4~(D)]{idealvalued} identifies
three explicit elements in the kernel of $p^*\colon H^i(BD_8)\to 
H^i(\bpm)$: one for each of $i=m+2$, $i=m+3$, and $i=m+4$. 
In particular, this produces at most four basis elements in 
the ideal Ker$(p^*)$ in dimensions at most $m+4$. However
we have just seen that, for $m+1\leq i\leq 2m-1$, the kernel of 
$p^*\colon H^i(BD_8)\to H^i(\bpm)$ is an $\F2$-vector space of dimension 
$i-m$. This means that through dimensions at most $m+4$ (and with $m>4$)
there are at least six more basis elements remaining
to be identified in Ker$(p^*)$.
}\end{ejemplo}

We next turn to the case when $m$ is odd
(a hypothesis in force throughout the rest of the section)
assuming, from Lemma~\ref{d2m1} on, that 
$m\equiv1\bmod4$.

\begin{nota}\label{elm1}{\em
Since the $\P^1$-case in Proposition~\ref{mono4} and 
Theorems~\ref{descripciondesaordenada} and~\ref{HF1m4} is elementary
(in view of Remark~\ref{m1} and Corollary~\ref{HBD}), we will implicitly 
assume $m\neq1$.
}\end{nota}

\medskip The CLSS of the $D_8$-action on $V_{m+1,2}$
now has a few extra complications that turn the analysis of 
differentials into a 
harder task. To begin with, we find a twisted system of 
local coefficients (Theorem~\ref{D8actionV}). As a $\mathbb{Z}[D_8]$-module, 
$H^{q}(V_{m+1,2})$ is:

\vspace{-3mm}
\begin{itemize}
\item $\mathbb{Z}$ for $q=0,m$;

\vspace{-3mm}
\item $\mathbb{Z}_\alpha$ for $q=m-1,2m-1$; 

\vspace{-3mm}
\item the zero module otherwise.
\end{itemize}

\noindent 
Thus, in total dimensions at most $2m-2$ the CLSS
is concentrated on the three horizontal lines with $q=0,m-1,m$.
[This is in fact the case in total dimensions at most $2m-1$, since
$H^0(BD_8;\mathbb{Z}_\alpha)=0$; this observation is not relevant for
the actual group $H^{2m-1}(B(\P^m,2))=\Z2$---given in
the second assertion in Proposition~\ref{orientabilidadB}---, 
but it will be relevant for the claimed surjectivity
of the map $p^*\colon H^{2m-1}(BD_8)\to H^{2m-1}(\bpm)$.]
In more detail, at the start of the CLSS
we have a copy of $H^*(BD_8)$ at
$q=0,m$, and a copy of $H^*(BD_8;\mathbb{Z}_\alpha)$ at $q=m-1$.
It is the extra horizontal 
line at $q=m-1$ (not present for an even $m$) that leads to 
potential $d_2$-differentials---from the ($q=m$)-line to the ($q=m-1$)-line.
Sorting these differentials out is the main difficulty (which we 
have been able to overcome
only for $m\equiv1\bmod4$). Throughout the remainder of the section we
work in terms of this spectral sequence, making free use of the
description of its $E_2$-term coming from Corollaries~\ref{HBD} and~\ref{HBDT},
as well as of its $H^*(BD_8)$-module structure. Note that the latter property
implies that much of the global structure of the spectral sequence
is dictated by differentials on the three elements
\begin{itemize}
\item $x_m\in E_2^{0,m}=H^0(BD_8;H^m(V_{m+1,2}))=
H^0(BD_8;\mathbb{Z})=\mathbb{Z}$;
\item $\alpha_1\in E_2^{1,m-1}=
H^1(BD_8;H^{m-1}(V_{m+1,2}))=H^1(BD_8;\mathbb{Z}_\alpha)=\mathbb{Z}_2$;
\item $\alpha_2\in E_2^{2,m-1}=H^2(BD_8;H^{m-1}(V_{m+1,2}))=H^2(BD_8;
\mathbb{Z}_\alpha)=\mathbb{Z}_4$;
\end{itemize}
each of which is 
a generator of the indicated group (notation is inspired by
that in Theorem~\ref{modulestructure} and in the proof of 
Theorem~\ref{D8actionV}---for even $n$).

\begin{lema}\label{d2m1}
For $m\equiv1\bmod4$ and
$m\geq 5$, the nontrivial $d_2$-differentials are
given by $d_2^{\hspace{.25mm}4i,m}(\kappa_4^i x_m)=2
\kappa_4^i\alpha_2$ for $i\geq0$.
\end{lema}
\begin{proof}
The only potentially nontrivial $d_2$-differentials
originate at the ($q=m$)-line and, in view of the module structure, all
we need to show is that
\begin{equation}\label{dif2}
d_2\colon E_2^{0,m}\to E_2^{2,m-1}\,\mbox{ has }\,d_2(x_m)=2\alpha_2
\end{equation}
(here and in what follows we omit superscripts of differentials). 

\medskip Let $m=4a+1$. Since $H^{2m-1}(\bpm)=\langle 1\rangle$, 
most of the elements in $E_2^{2m-1,0}=\langle 4a\rangle$ must be 
wiped out by differentials. The only differentials landing in a
$E_r^{2m-1,0}$ (that originate at a nonzero group) are
\begin{equation}\label{injdifs1}
d_m\colon E_m^{m-1,m-1}\to E_m^{2m-1,0}\mbox{ \ \ and \ \ }
d_{m+1}\colon E_{m+1}^{m-2,m}\to E_{m+1}^{2m-1,0}.
\end{equation}
But $E_2^{m-1,m-1}=\langle 2a\rangle$ and $E_{2}^{m-2,m}=\langle 2a-1\rangle$,
so that rank considerations imply 
\begin{equation}\label{cps1}
E_{2}^{m-2,m}=E_{m+1}^{m-2,m}, 
\end{equation}
with the two differentials in~(\ref{injdifs1})  
injective. In particular we get that
\begin{equation}\label{infty1}
H^{2m-1}(\bpm)=\langle 1\rangle \mbox{ comes from }
E_\infty^{2m-1,0}=\langle 1\rangle.
\end{equation}
Furthermore,~(\ref{cps1}) and
the $H^*(BD_8)$-module structure in the spectral sequence 
imply that the differential in~(\ref{dif2}) cannot be surjective.

\medskip
It remains to show that the differential in~(\ref{dif2}) is nonzero.
We shall obtain a contradiction by assuming that
$d_2(x_m)=0$, so that every element in the 
($q=m$)-line is a $d_2$-cycle.
Since $H^{2m}(\bpm)=0$, all of $E_2^{2m,0}=\langle 4a+2\rangle$ must be 
wiped out by differentials, and under the current hypothesis the only possible 
such differentials would be $d_m\colon E_m^{m,m-1}=E_2^{m,m-1}=\langle 
2a+1\rangle\to E_m^{2m,0}=E_2^{2m,0}$ and $d_{m+1}\colon E_{m+1}^{m-1,m}=
E_2^{m-1,m}=\langle 2a\rangle\oplus\mathbb{Z}_4\to 
E_{m+1}^{2m,0}\,$---indeed, $E_2^{0,2m-1}=H^0(BD_8;\mathbb{Z}_\alpha)=0$.
Thus, the former differential would have to be injective while the 
latter one would have to be surjective with a $\Z2$ kernel. But there 
are no further differentials that could kill the resulting 
$E_{m+2}^{m-1,m}=\langle1\rangle$, in contradiction to~(\ref{infty1}).
\end{proof}

\begin{nota}\label{criticals}{\em
In the preceding proof we  
made crucial use of the $H^*(BD_8)$-module 
structure in the spectral sequence in order to handle $d_2$-differentials. 
We show next that, just as in the proof of Theorem~\ref{HFpar} 
for $G=D_8$, many of the properties of all 
higher differentials in the case $m\equiv1\bmod4$ follow from
the ``size'' of the resulting $E_3$-term.
}\end{nota}

\begin{proof}[Proof of Theorem~{\em\ref{HF1m4}} 
for $G\hspace{.25mm}{=}\hspace{.35mm}D_8$,
and of Proposition~{\em\ref{mono4}}, both for 
$m\hspace{.5mm}{\equiv}\hspace{.5mm}1\bmod4$] 
The $d_2$-differentials in Lemma~\ref{d2m1} replace, by a 
$\Z2$-group, every instance of a $\mathbb{Z}_4$-group in the ($q=m-1$) and 
($q=m$)-lines of the $E_2$-term. This describes the $E_3$-term, the starting 
stage of the CLSS in the following considerations
(note that the $E_3$-term agrees with the $E_m$-term).
With this information the idea of the proof 
is formally the same as that in the case of an even $m$, 
namely: a little input from the 
Bockstein long exact sequence for $\bpm$ forces the injectivity of 
all relevant higher differentials
(we give the explicit details for the reader's benefit).

\smallskip
Let $m=4a+1$ (recall we are assuming $a\geq1$).
The crux of the matter is showing that the differentials 
\begin{equation}\label{difiny2}
d_m \colon E_3^{m-\ell,m-1}\to E_3^{2m-\ell,0}\;\mbox{ with }\;
\ell=0,1,2,\ldots,m
\end{equation}
and
\begin{equation}\label{difiny3}
d_{m+1} \colon E_3^{m-\ell-1,m}\to E_{m+1}^{2m-\ell,0}\;
\mbox{ with }\;\ell=0,1,2,\ldots,m-1
\end{equation}
are injective and never hit twice the generator of a $\mathbb{Z}_4$-group.
This assertion has already 
been shown for $\ell=1$ in the paragraph containing~(\ref{injdifs1}).
Likewise, the assertion for $\ell=0$ follows from~(\ref{infty1}) with
the same counting argument as the one 
used in the final paragraph of the proof of 
Lemma~\ref{d2m1}. Furthermore the case $\ell=m$ in~(\ref{difiny2}) is obvious 
since $E_3^{0,m-1}=H^0(BD_8;\mathbb{Z}_\alpha)=0$. However, since $E_3^{0,m}=
H^0(BD_8)=\mathbb{Z}$ and $E_3^{m+1,0}=H^{m+1}(BD_8)=\langle2a+2\rangle$,
the injectivity assertion needs to be suitably interpreted for $\ell=m-1$
in~(\ref{difiny3}); indeed, we will prove that  
\begin{equation}\label{sutilezainjectiva}
d_{m+1}\colon E_3^{0,m}\to E_{m+1}^{m+1,0} 
\mbox{ \ yields an injective map {\it after} tensoring with $\Z2$.}
\end{equation}
\indent From the $E_3$-term of the spectral sequence 
we easily see that $H^m(\bpm)$ is the direct 
sum of a copy of $\mathbb{Z}$ and a finite $2$-torsion group, while
$H^i(\bpm)$ is a finite $2$-torsion group for $i\neq0,m$.
We consider the analogue of~(\ref{spliteadas}), the 
short exact sequences
\begin{equation}\label{spliteadas2}
0\to\mbox{Coker}(2_{i})\to H^i(\bpm;\F2)\to\mbox{Ker}(2_{i+1})\to0,
\end{equation}
working here and below in the range $m+1\leq i\leq 2m-2$. 
Let $r_i$ denote the $2$-rank of (the torsion subgroup of) $H^i(\bpm)$,
so that Ker$(2_i)\cong{}$Coker$(2_i)\cong\langle r_i\rangle$.
Then Corollary~\ref{handel68B},~(\ref{spliteadas2}), 
and an easy induction (grounded by the fact that 
$\mbox{Ker}(2_{2m-1})=\langle1\rangle$, which in turn comes from 
the second assertion in 
Proposition~\ref{orientabilidadB}) yield that 
\begin{equation}\label{rangoagain}
\mbox{$r_{2m-\ell}$ is the integral part of 
$\frac{\ell+1}2$ for $2\leq\ell\leq m-1$}. 
\end{equation}
\indent Now, in the range of~(\ref{rangoagain}), Lemma~\ref{d2m1}
and Corollaries~\ref{HBD} and~\ref{HBDT} give
\begin{eqnarray*}
E_3^{m-\ell,m-1} & = & \begin{cases}
\left\langle 2a+1-\frac\ell2\right\rangle,&
\ell\mbox{ even};\\
\left\langle 2a-\frac{\ell-1}2\right\rangle,&\ell\mbox{ odd};
\end{cases}\\
E_3^{m-\ell-1,m} & = & \begin{cases}
\mathbb{Z}\,,&\ell=m-1;\\
\left\langle 2a+1-\frac\ell2\right\rangle,&
\ell\mbox{ even},\,\ell<m-1;\\
\left\langle 2a-\frac{\ell+1}2\right\rangle,&\ell\mbox{ odd};
\end{cases}\\
E_3^{2m-\ell,0} & = & \begin{cases}
\left\langle 4a+2-\frac\ell2\right\rangle,&\ell\equiv0\bmod4;\\
\left\{ 4a+1-\frac{\ell}2\right\}&
\ell\equiv2\bmod4;\\
\left\langle 4a-\frac{\ell-1}2\right\rangle,&\mbox{otherwise};
\end{cases}
\end{eqnarray*}
and since $E_{m+2}^{2m-\ell,0}$ has $2$-rank at most $r_{2m-\ell}$
(indeed, $E_{m+2}^{2m-\ell,0}=E_\infty^{2m-\ell,0}$ which is a subgroup of 
$H^{2m-\ell}(\bpm)$), an easy counting argument (using, as in the case
of an even $m$, the right exactness of the tensor product)
gives that the differentials in~(\ref{difiny2}) and~(\ref{difiny3}) must yield
an injective map after tensoring with $\Z2$. In particular they
\begin{itemize}
\item[(a)] must be injective on the nose, except for the case discussed 
in~(\ref{sutilezainjectiva});
\item[(b)] cannot hit twice the generator of a $\mathbb{Z}_4$-summand.
\end{itemize}

The already observed equalities $E_2^{0,2m-1}=H^0(BD_8;\mathbb{Z}_\alpha)=0$ 
together with~(a) above imply that, in total dimensions $t$ with $t\leq2m-1$ 
and $t\neq m$, the $E_{m+2}$-term of the spectral sequence is concentrated 
on the base line ($q=0$), while at higher lines ($q>0$)
the spectral sequence only has a $\mathbb{Z}$-group---at node $(0,m)$. 
This situation yields Theorem~\ref{HF1m4}, while~(b) above yields 
Proposition~\ref{mono4}.
\end{proof}

A direct calculation (left to the 
reader) using the proved behavior of the differentials in~(\ref{difiny2}) 
and~(\ref{difiny3})---and using (twice) the analogue of 
Proposition~\ref{algebraico} when $\ell\equiv2\bmod4\,$---gives
$$
H^{2m-\ell}(\bpm)=\begin{cases}
\left\langle\frac\ell2\right\rangle,&\ell\equiv0\bmod4;\\
\left\{ \frac\ell2-1 \right\}
,&\ell\equiv2\bmod4;\\
\left\langle\frac{\ell+1}2\right\rangle,&\mbox{otherwise;}
\end{cases}
$$
for $2\leq\ell\leq m-1$. Thus, as the reader can easily check using
Corollaries~\ref{HBD} and~\ref{HBDT}, instead of the symmetry 
isomorphisms exemplified in Table~\ref{tabla}, the cohomology groups 
of $\bpm$ are now formed (as predicted by the 
isomorphisms~(\ref{linkisos}) 
of the previous section) by a combination of $H^*(BD_8)$ and 
$H^*(BD_8;\mathbb{Z}_\alpha)$---in the lower and upper halves, respectively.
Once again, the CLSS analysis not only offers an alternative 
to the (torsion linking form) arguments in the previous section, 
but it allows us to recover, under the present hypotheses, the 
torsion subgroup in the three missing dimensions in~(\ref{faltantes}).

\begin{ejemplo}\label{muchkernel2}{\em
For $m\equiv1\bmod4$,~\cite[Theorem~1.4~(D)]{idealvalued} identifies
two explicit elements in the kernel of $p^*\colon H^i(BD_8)\to 
H^i(\bpm)$: one for each of $i=m+1$ and $i=m+3$. 
In particular, this produces at most three basis elements in 
the ideal Ker$(p^*)$ in dimensions at most $m+3$. However it follows from
the previous spectral sequence analysis that, for $m+1\leq i\leq 2m-1$, 
the kernel of $p^*\colon H^i(BD_8)\to H^i(\bpm)$ is an $\F2$-vector space 
of dimension $i-m+(-1)^i$. 
This means that through dimensions at most $m+3$ (and with $m\geq5$)
there are at least four more basis elements remaining 
to be identified in Ker$(p^*)$.
}\end{ejemplo}

\section{Case of $B(\P^{4a+3},2)$}\label{problems3}
We now discuss some aspects of the spectral sequence of the previous section
in the unresolved case $m\equiv3\bmod4$.
Although we are unable to describe the pattern of differentials
for such $m$, we show that 
enough information can be collected to not only resolve 
the three missing cases in~(\ref{faltantes}), but 
also to conclude the proof of Theorem~\ref{HF1m4} for $G=D_8$.
Unless explicitly stated otherwise, 
the hypothesis $m\equiv3\bmod4$ will be in force throughout the section.

\begin{nota}\label{lam3}{\em
The main problem that has prevented us from fully 
understanding the spectral sequence of this section comes from the 
apparent fact that the algebraic input coming from the 
$H^*(BD_8)$-module structure in the CLSS---the crucial 
property used in the proof of Lemma~\ref{d2m1}---does not
give us enough information in order to determine the 
pattern of $d_2$-differentials. New geometric insights 
seem to be needed instead. Although 
it might be tempting to conjecture the validity
of  Lemma~\ref{d2m1} for $m\equiv 3\bmod 4$, we have not found 
concrete evidence supporting such a possibility. 
In fact, a careful analysis of the possible
behaviors of the spectral sequence for $m=3$ (performed in Example~6.4 of
the preliminary version \cite{v1} of this paper) does not give even a more
aesthetically pleasant reason for leaning toward the possibility of having a
valid Lemma~\ref{d2m1} in the current congruence. A second problem 
arose in~\cite{v1} when we noted that,
even if the pattern of $d_2$-differentials were known for 
$m\equiv3\bmod4$, there would seem to be a slight indeterminacy 
either in a few higher differentials (if Lemma~\ref{d2m1} holds for 
$m\equiv3\bmod4$), or in a few possible extensions among the $E_\infty^{p,q}$ 
groups (if Lemma~\ref{d2m1} actually fails for $m\equiv3\bmod4$).
Even though we cannot resolve
the current $d_2$-related ambiguity, in~\cite[Example~6.4]{v1} we note that,
at least for $m=3$, it is possible to overcome the above mentioned problems 
about higher differentials or possible extensions 
by making use of the explicit description of $H^4(B(\P^3,2))$---given later
in the section (considerations previous to Remark~\ref{expander}) 
in regard to the claimed surjectivity 
of~(\ref{excepcional}); see also~\cite{taylor}, 
where advantage is taken of the fact 
that $\P^3$ is a group. }\end{nota}

In the first result of this section, 
Theorem~\ref{HF1m4} for $G=D_8$ and $m\equiv3\bmod4$, we show that,
despite the previous comments, the spectral sequence approach 
can still be used to compute $H^*(B(\P^{4a+3},2))$ just beyond the middle 
dimension (i.e., just before the first problematic $d_2$-differential 
plays a decisive role). In particular, this computes the corresponding 
groups in the first two of the three missing cases in~(\ref{faltantes}).

\begin{proposicion}\label{dimm}
Let $m=4a+3$. The map $H^i(BD_8)\to H^i(\bpm)$ 
induced by~{\em(\ref{lap})} is:
{\em \begin{enumerate}
\item {\em an isomorphism for $i<m;$}
\item {\em a monomorphism onto the torsion subgroup of $H^i(\bpm)=
\langle 2a+1\rangle\oplus\mathbb{Z}$ for $i=m;$}
\item {\em the zero map for $\,2m-1<i$.}
\end{enumerate}}
\end{proposicion}
\begin{proof}
The argument parallels that used in the analysis of the CLSS when
$m\equiv1\bmod4$. Here is the chart of the current $E_2$-term through 
total dimensions at most $m+1$:

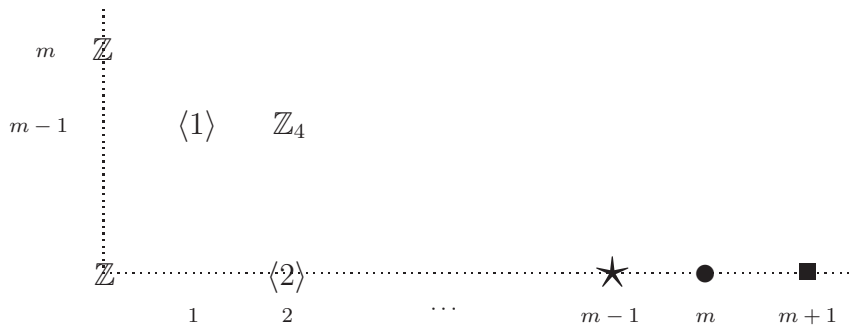
\begin{picture}(0,91)(-50,-16)
\qbezier[100](0,0)(100,0)(200,0) 
\qbezier[40](0,0)(0,35)(0,70) 
\multiput(-2.5,-3)(25,0){1}{$\mathbb{Z}$}
\put(22.5,-13){\scriptsize$1$}
\put(47.5,-13){\scriptsize$2$}
\put(-25,37.5){\scriptsize$m-1$}
\put(-18,57.5){\scriptsize$m$}
\put(43.5,-3){$\langle2\rangle$}
\put(19.5,37.1){$\langle1\rangle$}
\put(45,37.1){$\mathbb{Z}_4$}
\put(-3,56.8){$\mathbb{Z}$}
\put(87,-11){\scriptsize$\cdots$}
\put(127,-13){\scriptsize$m-1$}
\put(157.5,-13){\scriptsize$m$}
\put(179.5,-13){\scriptsize$m+1$}
\put(131,-4){\Huge$\star$}
\put(157,-3){\Large$\bullet$}
\put(185,-1.7994){$\rule{2.2mm}{2.2mm}$}
\end{picture}

\noindent The star at node $(m-1,0)$ stands for $\langle2a+2\rangle$;
the bullet at node $(m,0)$ stands for $\langle2a+1\rangle$; 
the solid box at node $(m+1,0)$ stands for 
$\{2a+2\}$. In this range there are only three possibly nonzero
differentials: 
\begin{itemize}
\item a $d_2$ from node $(0,m)$ to node $(2,m-1)$; 
\item a $d_{m}$ from node $(1,m-1)$ to node $(m+1,0)$;
\item a $d_{m+1}$ from node $(0,m)$ to node $(m+1,0)$.
\end{itemize}
Whatever these $d_2$ and $d_{m+1}$ are, there will be a resulting 
$E_\infty^{0,m}=\mathbb{Z}$. On the other hand, the argument about 2-ranks 
in~(\ref{spliteadas}) and in~(\ref{spliteadas2}), leading respectively 
to~(\ref{ranks}) and~(\ref{rangoagain}), now yields that the torsion 2-group
$H^{m+1}(\bpm)$ has 2-rank $2a+1$. Since $E_\infty^{m+1,0}$ is a subgroup of
$H^{m+1}(\bpm)$, this forces the two differentials $d_{m}$ and $d_{m+1}$ above 
to be nonzero, each one with cokernel of 2-rank one less than the 2-rank of its 
codomain. In fact, $d_{m}$ must have cokernel isomorphic to $\{2a+1\}$, 
whereas the cokernel of $d_{m+1}$ is either $\{2a\}$ or $\langle2a+1\rangle$
(Remark~\ref{expander}, and especially~\cite[Example~6.4]{v1}, 
expand on these 
possibilities). What matters here is the forced injectivity of $d_m$, which 
implies $E_\infty^{1,m-1}=0$ and, therefore, the second assertion of the 
proposition---the first assertion is obvious from the 
CLSS, while the third 
one is elementary.
\end{proof}

We now start work on the only groups in 
Theorem~\ref{descripciondesaordenada} not 
yet computed, namely $H^{m+1}(\bpm)$ for $m=4a+3$. 
As indicated in the previous proof, 
these are torsion 2-groups of 2-rank $2a+1$. 
Furthermore,~(\ref{excepcional}) and Corollary~\ref{HBDT}
show that each such group contains a copy of $\{2a\}$,
a 2-group of the same 2-rank as that of $H^{m+1}(\bpm)$. In showing that
the two groups actually agree (thus completing the proof of 
Theorem~\ref{descripciondesaordenada}), a key fact 
comes from Fred Cohen's observation 
(recalled in the paragraph previous to Remark~\ref{m1}) that {\it there 
are no elements of order $8$}. For instance,
\begin{equation}\label{casoespecial}
\mbox{when $m=3$ the two groups must
agree since both are cyclic (i.e., have 2-rank 1). } 
\end{equation}
In order to deal with the situation for positive values of $a$, 
Cohen's observation is coupled with 
a few computations in the first two pages of 
the Bockstein spectral sequence (BSS) for $\bpm$: we will show that 
there is only one copy of $\mathbb{Z}_4$ (the one coming from the 
subgroup $\{2a\}$) in the decomposition of $H^{m+1}(\bpm)$ as a sum of 
cyclic 2-groups---forcing $H^{m+1}(\bpm)=\{2a\}$. 

\begin{nota}\label{expander}{\em
Before undertaking the BSS calculations (in Proposition~\ref{sq1} below), 
we pause to observe that, unlike the Bockstein input in all the 
previous CLSS-related proofs, the use of the BSS
does not seem to give quite enough information in order to understand the
pattern of $d_2$-differentials in the current CLSS. Much of the problem
lies in being able to decide the actual cokernel of the 
$d_{m+1}$-differential in the previous proof and, consequently,
understand how the $\mathbb{Z}_4$-group in $H^{m+1}(\bpm)$ arises
in the current CLSS; either entirely at the $q=0$ line 
(as in all cases of the previous---and the next---section), or as a nontrivial 
extension in the $E_\infty$ chart (just as in the case of the 
spectral sequence in Section~\ref{STCP67}\hspace{.3mm}---see 
also~\cite[Section~9]{taylor}). 
}\end{nota}

Recall from~\cite{feder,handel68} that the mod 2 cohomology ring of 
$\bpm$ is polynomial on three classes $x$, $x_1$, and $x_2$, of respective 
dimensions 1, 1, and 2, subject to the three relations
\begin{itemize}
\item[(I)] $\quad x^2=xx_1$;
\item[(II)] $\quad\hspace{1.5mm} \displaystyle\sum_{0\leq i\leq \frac{m}{2}}
\binom{m-i}{i}x_1^{m-2i}x_2^i=0$;
\item[(III)] $\quad \displaystyle\sum_{0\leq i\leq \frac{m+1}{2}}\binom{m+1-i}{i}
x_1^{m+1-2i}x_2^i=0$.
\end{itemize}
Further, the action of $\sq1$ is determined by (I) and
\begin{equation}\label{sq1y}
\sq1 x_2=x_1x_2.
\end{equation}
[The following observations---proved 
in~\cite{feder,handel68}, but not needed in this paper---might help 
the reader to assimilate the facts just described: 
The three generators $x$, $x_1$, and $x_2$
are in fact the images under the map $p_{m,D_8}$ 
in~(\ref{lap}) of the corresponding classes at the beginning of 
Section~\ref{HBprelim}. In turn, the latter generators $x_1$ and $x_2$
come from the Stiefel-Whitney classes $w_1$ and $w_2$ in $B\mathrm{O}(2)$
under the classifying map for the inclusion $D_8\subset\mathrm{O}(2)$. 
In these terms,~(\ref{sq1y}) corresponds to the (simplified in $B\mathrm{O}
(2)$) Wu formula $\sq1(w_2)=w_1w_2$. Finally, the two relations (II) and 
(III) correspond to the fact that the two dual Stiefel-Whitney classes
$\overline{w}_m$ and $\overline{w}_{m+1}$ in $B\mathrm{O}(2)$ generate the
kernel of the map induced by the Grassmann inclusion $G_{m+1,2}\subset 
B\mathrm{O}(2)$.] 

\medskip
Let $R$ stand for the subring generated by $x_1$ 
and $x_2$, so that there is an additive splitting
\begin{equation}\label{splitting}
H^*(\bpm;\F2)=R\oplus x\cdot R
\end{equation}
which is compatible with the action of $\sq1$ (note that multiplication by $x$
determines an additive isomorphism $R\cong x\cdot R$).

\begin{proposicion}\label{sq1}
Let $m=4a+3$. With respect to the differential $\sq1:$
\begin{itemize}
\item $H^{m+1} (R; \sq1)=\Z2$.
\item $H^{m+1} (x\cdot R; \sq1) = 0$.
\end{itemize}
\end{proposicion}

Before proving this result, let us indicate how it can be used to show 
that~(\ref{excepcional}) is an isomorphism for $m=4a+3$. As explained in 
the paragraph containing~(\ref{casoespecial}), we must have
\begin{equation}\label{laerre}
2\cdot H^{4a+4}(B(\P^{4a+3},2))=\langle r\rangle\quad\mbox{with}\quad r\geq1
\end{equation}
and we need to show that $r=1$ is in fact the case. 
Consider the Bockstein exact couple 

\medskip\begin{picture}(0,50)(0,-37)
\put(33.5,0){$H^*(B(P^{4a+3},2))$}
\put(96,3){\vector(1,0){110}}
\put(150,6){\scriptsize$2$}
\put(208,0){$H^*(B(P^{4a+3},2))$}
\put(207,-4){\vector(-2,-1){30}}
\put(195,-15){\scriptsize$\rho$}
\put(113,-30){$H^*(B(P^{4a+3},2);\hbox{$\mathbb{F}$}_2).$}
\put(120,-20){\vector(-2,1){30}}
\put(100,-18){\scriptsize$\delta$}
\end{picture}

\noindent
In the  (unravelled) derived exact couple
\begin{eqnarray*}
\lefteqn{\cdots \to2\cdot H^{4a+4}(B(P^{4a+3},2)) \stackrel{2}{\rightarrow} 2\cdot H^{4a+4} (B(P^{4a+3},2))\to\quad\ }\\[2mm]
&&\qquad\rightarrow H^{4a+4} (H^*(B(P^{4a+3},2); \mathbb{F}_2);\sq1)\rightarrow 2\cdot H^{4a+5} (B(P^{4a+3},2))\rightarrow \cdots
\end{eqnarray*}
we have $2\cdot H^{4a+5}(B(P^{4a+3},2))=0$ since $H^{4a+5} (B(P^{4a+3},2))=
\langle 2a+1\rangle$---argued in Section~\ref{linkinsection} by means of the 
(twisted) torsion linking form. Together with~(\ref{laerre}), this implies 
that the map 
\begin{equation}\label{lacentral}
\langle r\rangle=2\cdot H^{4a+4} (B(P^{4a+3},2))\to 
H^{4a+4} (H^*(B(P^{4a+3},2);\mathbb{F}_2);\sq1)
\end{equation}
in the above exact sequence is 
an isomorphism. Proposition~\ref{sq1} and~(\ref{splitting}) then imply the 
required conclusion $r=1$.

\begin{proof}[Proof of Proposition~{\em\ref{sq1}}] Note that every binomial 
coefficient in (II) with $i\not\equiv0\bmod4$ is congruent to zero mod $2$. 
Therefore relation (II) can be rewritten as
\begin{equation}
\label{laI}
x_1^{4a+3} = \sum^{a/2}_{j=1} \binom{a-j}{j}x_1^{4(a-2j)+3} x_2^{4j}.
\end{equation}
Likewise, every binomial coefficient in (III) with $i\equiv 3\bmod4$ is 
congruent to zero mod $2$. Then, taking into account~(\ref{laI}), relation 
(III) becomes
\begin{eqnarray}
\label{laII}
x_2^{2a+2} &=& x_1^{4a+4} +\sum_{i\in\Lambda} 
\binom{4a+4-i}{i}x_1^{4a+4-2i} x_2^i \nonumber\\
&=& \sum^{a/2}_{j=1} \binom{a-j}{j}x_1^{4(a-2j)+4} x_2^{4j} +
\sum_{i\in\Lambda}\binom{4a+4-i}{i}x_1^{4a+4-2i}x_2^i
\end{eqnarray}
where $\Lambda$ is the set of integers $i$ with $1\leq i\leq2a+1$ and 
$i\not\equiv3\bmod4$. Using (\ref{laI}) and (\ref{laII}) it is a simple 
matter to write down a basis for $R$ and $x\cdot R$ in dimensions $4a+3$, 
$4a+4$, and $4a+5$. The information is summarized (under the assumption $a>0$,
which is no real restriction in view of~(\ref{casoespecial}))
in the following chart,  
where elements in a column form a basis in the indicated dimension, and where 
crossed out terms can be expressed as linear combination of the other ones 
in view of~(\ref{laI}) and~(\ref{laII}). 

{\unitlength=1cm\scriptsize
\begin{picture}(0,9.2)(0,-7.6)
\put(0,.65){\normalsize$4a+3$}
\put(0,0){$x_1^{4a+3}$}
\put(-.05,-.1){\line(2,1){.9}}
\put(0,-.5){$x_1^{4a+1}x_2$}
\put(0,-1){$x_1^{4a-1}x_2^2$}
\put(0,-1.5){$x_1^{4a-3}x_2^3$}
\put(.5,-2){$\vdots$}
\put(0,-2.5){$x_1^3x_2^{2a}$}
\put(0,-3){$x_1x_2^{2a+1}$}

\put(1.5,-.4){\vector(1,0){1}}
\put(2.73,-.5){$0$}
\put(1.5,-1.4){\vector(1,0){1}}
\put(2.73,-1.5){$0$}
\put(1.5,-2.9){\vector(1,0){1}}
\put(2.73,-3){$0$}
\put(1.5,-.9){\vector(1,0){5}}
\put(1.5,-2.4){\vector(1,0){5}}
\put(8.5,-1.4){\vector(1,0){4.5}}
\put(8.5,-2.9){\vector(1,0){4.5}}

\put(7,.65){\normalsize$4a+4$}
\put(7,0){$x_1^{4a+4}$}
\put(6.95,-.1){\line(2,1){.9}}
\put(6.95,-3.6){\line(2,1){.9}}
\put(7,-.5){$x_1^{4a+2}x_2$}
\put(7,-1){$x_1^{4a}x_2^2$}
\put(7,-1.5){$x_1^{4a-2}x_2^3$}
\put(7.5,-2){$\vdots$}
\put(7,-2.5){$x_1^4x_2^{2a}$}
\put(7,-3){$x_1^2x_2^{2a+1}$}
\put(7,-3.5){$x_2^{2a+2}$}

\put(13.5,.65){\normalsize$4a+5$}
\put(13.5,0){$x_1^{4a+5}$}
\put(13.45,-.1){\line(2,1){.9}}
\put(13.5,-.5){$x_1^{4a+3}x_2$}
\put(13.45,-.6){\line(4,1){1.2}}
\put(13.45,-3.57){\line(4,1){1.2}}
\put(13.5,-1){$x_1^{4a+1}x_2^2$}
\put(13.5,-1.5){$x_1^{4a-1}x_2^3$}
\put(13.5,-2){$x_1^{4a-3}x_2^4$}
\put(14,-2.5){$\vdots$}
\put(13.5,-3){$x_1^3x_2^{2a+1}$}
\put(13.5,-3.5){$x_1x_2^{2a+2}$}
\multiput(0,-3.85)(.1,0){147}{\circle*{.05}}
\end{picture}

\begin{picture}(0,0)(0,-7.6)
\put(0,-4){$xx_1^{4a+2}$}
\put(0,-4.5){$xx_1^{4a}x_2$}
\put(0,-5){$xx_1^{4a-2}x_2^2$}
\put(.5,-5.8){$\vdots$}
\put(0,-6.5){$xx_1^2x_2^{2a}$}
\put(0,-7){$xx_2^{2a+1}$}

\put(1.5,-4.9){\vector(1,0){5}}
\put(1.5,-4.4){\vector(1,0){1}}
\put(1.5,-6.4){\vector(1,0){5}}
\put(2.73,-4.5){$0$}
\put(1.5,-6.9){\vector(1,0){1}}
\put(2.73,-7){$0$}
\put(8.5,-4.4){\vector(1,0){4.5}}
\put(8.5,-5.4){\vector(1,0){4.5}}
\put(8.5,-6.9){\vector(1,0){4.5}}

\put(6.95,-4.07){\line(4,1){1.2}}
\put(7,-4){$xx_1^{4a+3}$}
\put(7,-4.5){$xx_1^{4a+1}x_2$}
\put(7,-5){$xx_1^{4a-1}x_2^2$}
\put(7,-5.5){$xx_1^{4a-3}x_2^3$}
\put(7.5,-6){$\vdots$}
\put(7,-6.5){$xx_1^3x_2^{2a}$}
\put(7,-7){$xx_1x_2^{2a+1}$}

\put(13.45,-4.07){\line(4,1){1.1}}
\put(13.45,-7.57){\line(4,1){1.1}}
\put(13.5,-4){$xx_1^{4a+4}$}
\put(13.5,-4.5){$xx_1^{4a+2}x_2$}
\put(13.5,-5){$xx_1^{4a}x_2^2$}
\put(13.5,-5.5){$xx_1^{4a-2}x_2^3$}
\put(14,-6.3){$\vdots$}
\put(13.5,-7){$xx_1^2x_2^{2a+1}$}
\put(13.5,-7.5){$xx_2^{2a+2}$}
\end{picture}}

\noindent 
The top and bottom portions of the chart (delimited
by the horizontal dotted line) correspond to $R$ and 
$x\cdot R$, respectively. Horizontal arrows indicate 
$\sq1$-images, which are easily computable 
from~(\ref{sq1y}) and (I): $\sq1(x^i x_1^{i_1}x_2^{i_2})=0$ when 
$i+i_1+i_2$ is even, while $\sq1(x^i x_1^{i_1}x_2^{i_2})=
x^i x_1^{i_1+1}x_2^{i_2}$ when $i+i_1+i_2$ is odd---here $i\in\{0,1\}$ in
view of (I) above. There are only two basis elements,
in dimensions $4a+3$ and $4a+4$, whose 
$\sq1$-images are not indicated in the chart: $xx_1^{4a+2}\in (x\cdot R)^{4a+3}$ 
and $x_1^{4a+2}x_2 \in R^{4a+4}$. The second conclusion in the proposition is 
evident from the bottom part of the chart---no matter what the $\sq1$-image of 
$xx_1^{4a+2}$ is. On the other hand, the top portion of the 
chart implies that, in dimension
$4a+4$, Ker$(\sq1)$ and Im$(\sq1)$ are elementary $2$-groups whose ranks satisfy
$$
\mathrm{rk}(\mathrm{Ker} (\sq1)) =\mathrm{rk}(\mathrm{Im}(\sq1))+\varepsilon
$$
with $\varepsilon =1$ or $\varepsilon =0$ (depending on whether or not 
$\sq1(x_1^{4a+2}x_2)$ can be  written down as a linear combination of the 
elements $x_1^{4a-1}x_2^3$, $x_1^{4a-5}x_2^5,\ldots,$ and $x_1^3 
x_2^{2a+1}$---this of course depends on the actual binomial coefficients 
in~(\ref{laI})). But the possibility $\varepsilon =0$ is ruled out 
by~(\ref{laerre}) and~(\ref{lacentral}), forcing $\varepsilon=1$ and, 
therefore, the first assertion of this proposition.
\end{proof}
\section{Case of $F(\P^m,2)$}\label{HF}

The CLSS analysis in the previous two sections can be
applied---with $G=\Z2\times\Z2$ instead of $G=D_8$---in 
order to study the cohomology groups 
of the ordered configuration space 
$F(\P^m,2)$. The explicit details are similar but much easier
than those for unordered configuration spaces, and this time
the additive structure of differentials can be fully understood for any $m$. 
Here we only review the main differences, simplifications, and results.

\medskip
For one, there is no $4$-torsion to deal with (e.g.~the arithmetic 
Proposition~\ref{algebraico} is not needed); indeed, the role of $BD_8$ in 
the situation of an unordered configuration space $\bpm$ is played by 
$\P^\infty\times\P^\infty$ for ordered configuration spaces $F(\P^m,2)$. 
Thus, the use of Corollaries~\ref{HBD} and~\ref{HBDT} is
replaced by the simpler Lemma~\ref{kunneth}. But the most important 
simplification in the calculations relevant to the
present section comes from the absence of 
problematic $d_2$-differentials, the obstacle 
that prevented us from computing the CLSS of the $D_8$-action on 
$V_{m+1,2}$ for $m\equiv3\bmod4$. 
[This is why in Lemma~\ref{kunneth} we do not insist on describing 
$H^*(\P^\infty\times\P^\infty;\mathbb{Z}_\alpha)$ as a module over
$H^*(\P^\infty\times\P^\infty)$---compare to Remark~\ref{criticals}.]
As a result, the integral cohomology CLSS of the 
$(\Z2\times\Z2)$-action on $V_{m+1,2}$ can be fully understood,
without restriction on $m$, by means of the 
counting arguments used in 
Section~\ref{HB}, now forcing the injectivity of all relevant differentials
from the following two ingredients:
\begin{itemize}
\item[(a)] The size and distribution of the groups in the CLSS. 
\item[(b)] The $\Z2\times\Z2$ analogue of Proposition~\ref{orientabilidadB} 
in Remark~\ref{analogousversionsF}---the input triggering the determination
of differentials.
\end{itemize}
In particular, when $m$ is odd, the $\Z2\times\Z2$ analogue of  
Lemma~\ref{d2m1} does not arise and, instead, only
the counting argument in the proof following Remark~\ref{criticals}
is needed.

\medskip We leave it for the reader 
to supply details of the above CLSS and verify that 
this leads to Theorems~\ref{HFpar} and~\ref{HF1m4} in the case 
$G=\Z2\times\Z2$, as well as to the computation of all the 
cohomology groups in Theorem~\ref{descripcionordenada}.

\section{The symmetric topological complexity of $\P^5$ and $\P^6$}
\label{STCP67}

In this final section we use the cohomological information gathered in previous sections in order to compute the symmetric topological complexity of $\P^m$ for $m=5,6$ (Theorem~\ref{STC}). The method is an extension of that used in~\cite{taylor} to deal with the case $m=3$. 

\begin{definicion}\label{reducida}{\em
The topological complexity of a space $X$, $\TC(X)$, is 
defined as the \emph{reduced} Schwarz genus of the endpoints evaluation 
map $\ev\colon X^{[0,1]}\to X\times X$, $\ev(\gamma)=(\gamma(0),\gamma(1))$, 
i.e.~$\TC^S(X)+1$ gives the smallest cardinality of covers of $X\times X$ 
by open sets over each of which $\ev$ admits a (continuous) section. To 
define a symmetric version of $\TC(X)$, note that the involution on 
$X\times X$ that switches coordinates is compatible, via $\ev$, with the 
involution on $X^{[0,1]}$ that reverses a path. These actions are free
on the domain and codomain of the restricted fibration $\ev\colon 
\ev^{-1}(F(X,2))\to F(X,2)$ which thus, at the level of orbit spaces, 
yields a fibration $\ev'\colon \ev^{-1}(F(X,2))/\Z2\to B(X,2)$. The 
symmetric topological complexity of $X$, $\TC^S(X)$, is defined to be 
one less than the reduced Schwarz genus of $\ev'$.
}\end{definicion}

\begin{nota}\label{normalizacionTC}{\em
The adjustment by one in the definition of $\TC^S(X)$ 
does not come from any normalization convention---it can be thought 
of as accounting for the obvious symmetric section of $\ev$ over the 
(removed) diagonal.
Instead, the normalization we have taken for the Schwarz genus means that, 
just as in~\cite{taylor}, the values of $\TC(X)$ and $\TC^S(X)$ in this paper 
are one less than those originally defined in~\cite{F1,FGsymm}.
}\end{nota}

Before getting into the main technical computation of this section, 
it is convenient to set Theorem~\ref{STC} in context. The inequality 
\begin{equation}\label{inequalityX}
\TC^S(X)-\TC(X)\geq0
\end{equation}
is proved in~\cite[Corollary~9]{FGsymm} for any space $X$. It is optimal since, as proved in~\cite{symmotion},~(\ref{inequalityX}) becomes an equality when $X$ is, for instance, a complex projective space. However, as discussed in~\cite[Example~3.3]{symmotion}, there is no current indication that the left hand side in~(\ref{inequalityX}) should even be a bounded function of $m$ for $X=\P^m$. We discuss the known situation (as updated by Theorem~\ref{STC}) for a few particular families of $m$. In the following paragraph we use~\cite{tablas,FTY} as the main references for the known numerical values of $\TC(\P^m)$.

\medskip
To begin with, Example~3.3 in~\cite{symmotion} observes that $$\TC^S(\P^{2^i})-\TC(\P^{2^i})=1$$ for any $i\geq0$ (the case $i=0$ was not mentioned in~\cite{symmotion}, but it is covered by the calculations in~\cite{F1,FGsymm}). Example~3.3 in~\cite{symmotion} also notes that $$\TC^S(\P^{2^i+1})-\TC(\P^{2^i+1})=2$$ for any $i\geq3$; the corresponding result for $i=1,2$ is also true in view of~\cite{taylor} (for $i=1$) and Theorem~\ref{STC} (for $i=2$). Lastly, Example~3.3 in~\cite{symmotion} remarks that
\begin{equation}\label{laexcepcion}
\TC^S(\P^{2^i+2})-\TC(\P^{2^i+2})=1
\end{equation}
for any $i\geq4$. Now, while~(\ref{laexcepcion}) is also true for $i=3$ (as remarked in~\cite[Table~1]{taylor}), Theorem~\ref{STC} implies that, for $i=2$,~(\ref{laexcepcion}) must be replaced by $\TC^S(\P^{6})-\TC(\P^{6})=2$.

\bigskip
We now start working toward the proof of Theorem~\ref{STC}. As recalled in the Introduction, for any $m\geq1$, $\TC^S(\P^m)$ agrees with the smallest positive integer $n=n(m)$ for which the map in~(\ref{classify}) can be homotopy compressed into $\P^{n-1}$. We take advantage of the obvious inequality $n(m)\leq n(m+1)$: since $n(6)\leq 9$~(\cite[Corollary~11]{reesodd}), Theorem~\ref{STC} will follow once we show that the case $m=5$ of the map $u$ in~(\ref{classify}) cannot be homotopy compressed into $\P^7$. We prove in fact:

\begin{teorema}\label{altura}
The nonzero element $z\in H^2(\P^\infty)$ satisfies $u^*(z)^4\neq0$.
\end{teorema}

Our proof of Theorem~\ref{altura} is based on a direct study of the CLSS for the $\Z2$-action on $F(\P^5,2)$ in Definition~\ref{reducida} which, by definition, is classified by $u$. So, our first goal---accomplished in Proposition~\ref{goal1} below---is to describe the (highly) twisted coefficients of this spectral sequence, i.e.~the action in integral cohomology of the involution on $F(\P^5,2)$ that switches coordinates. 

\smallskip
The $(\Z2\times\Z2)$-action on $V_{m+1,2}$ given in Definition~\ref{inicio1Handel} extends to the standard product action of $\Z2\times\Z2$ on $S^\infty\times S^\infty$. Thus the sequence of $(\Z2\times\Z2)$-equivariant inclusions $V_{m+1,2}\hookrightarrow S^m\times S^m\hookrightarrow S^\infty\times S^\infty$ shows that the map $p=p_{m,\Z2\times\Z2}\colon F(\P^m,2)\to\P^\infty\times\P^\infty$ in~(\ref{lap}) factors (up to homotopy) as
\begin{equation}\label{lasinclusiones}
F(\P^m,2)\hookrightarrow\P^m\times\P^m\hookrightarrow\P^\infty\times\P^\infty.
\end{equation}
This fact is used in the proof of the following mod 2 result, which was brought to the authors' attention by Fred Cohen. Recall the cohomology classes $x_1,y_1\in H^*(\P^\infty\times\P^\infty;\F2)$ and $x_2,y_2,z_3\in H^*(\P^\infty\times\P^\infty)$ introduced in the paragraph containing~(\ref{relacionesenteras}).

\begin{lema}\label{relacionesmod2}
The morphism $p^*\colon H^*(\P^\infty\times\P^\infty;\F2)\to H^*(F(\P^m,2);\F2)$ is surjective with kernel the ideal generated by the three elements $x_1^{m+1}$, $y_1^{m+1}$, and $\,\sum_{i+j=m}x_1^iy_1^j$.
\end{lema}
\begin{proof} The first two elements generate the kernel of the second inclusion in~(\ref{lasinclusiones}), whereas the third element maps to the diagonal cohomology class in $\P^m\times\P^m$ in view of~\cite[Theorem~11.11]{MS}---which certainly restricts to zero in $F(\P^m,2)$. So it suffices to check that the first inclusion in~(\ref{lasinclusiones}) is surjective with kernel generated by the diagonal class. But~\cite[Section~11]{MS} embeds the map under consideration into a long exact sequence
$$
\cdots\to H^{*-m}(P^m;\Z2)\to H^*(\P^m\times\P^m;\Z2)\to H^*(F(\P^m,2);\Z2)\to\cdots
$$
(written here in terms of the Thom isomorphism for the normal bundle of the diagonal inclusion $\P^m\hookrightarrow\P^m\times\P^m$). The desired conclusion then follows from~\cite[Lemma~11.8]{MS} which shows that the map of degree $m$ in this long exact sequence is given by multiplication by the diagonal class $\sum_{i+j=m}x_1^iy_1^j$---clearly a monomorphism in the current case.
\end{proof}

The argument in the previous proof cannot be applied with integer coefficients for a non-orientable projective space (or manifold, for that matter). Nevertheless we prove:

\begin{corolario}\label{estructuramultiplicativaCOR}
The kernel of $p^*\colon H^*(\P^\infty\times\P^\infty)\to H^*(F(\P^5,2))$ is the ideal generated by the three elements $x_2^3$, $y_2^3$, and $z_3(x_2^2+x_2y_2+y_2^2)$. Further, an $\F2$-basis for the torsion groups in $H^*(F(\P^5,2))$ is given by the {\em({\em$p^*$-images of the})} elements in Table~{\em\ref{labase}}.
\end{corolario}

\begin{table}[h]
\centerline{
\begin{tabular}{|c|c|c|c|c|c|c|c|c|c|c|}
\hline
\rule{0pt}{15pt}$*={}$& 0 & 1 & 2 & 3 & 4 & 5 & 6 & 7\\[0.5ex]
\hline
\rule{0pt}{15pt} &---&---&$x_2,\,\;y_2$&$z_3$&$x_2^2,\,\;x_2y_2,\,\;y_2^2$&$x_2z_3,\,\;y_2z_3$&$x_2^2y_2,\,\;x_2y_2^2 $&$x_2^2z_3,\,\;y_2^2z_3$\\[0.5ex]
\hline
\end{tabular}}
\caption{Basis elements for $TH^*(F(\P^5,2))$ through $*\leq7$
\label{labase}}
\end{table}
\begin{proof}
The $\P^5$-case of Theorem~\ref{descripcionordenada} implies that the mod 2 reduction map $H^*(F(\P^5,2))\to H^*(F(\P^5,2);\F2)$ is injective in positive dimensions not 5, so that a straightforward calculation using Remark~\ref{mapdereduccion} and Lemma~\ref{relacionesmod2} yields that the three indicated classes lie in the kernel of $p^*$. The result then follows from an easy counting argument, taking into account~(\ref{relacionesenteras}), Theorem~\ref{HF1m4} (for $m=5$ and $G=\Z2\times\Z2$), and the full description of $H^*(F(\P^5,2))$ in Theorem~\ref{descripcionordenada}.
\end{proof}

As suggested by Corollary~\ref{estructuramultiplicativaCOR}, it will be convenient to denote elements in the torsion groups of $H^*(F(\P^5,2))$ by their corresponding preimages in Table~\ref{labase}. Next we choose a generator of the torsion-free summand in $H^5(F(\P^5,2))$. Since $H^6(F(\P^5,2))$ is an $\F2$-vector space, the image of the reduction map $H^5(F(\P^5,2))\to H^5(F(\P^5,2);\F2)$ agrees with the kernel of $\mathrm{Sq}^1\colon H^5(F(\P^5,2);\F2)\to H^6(F(\P^5,2);\F2)$. The latter is easily seen to have 
\begin{equation}\label{tresmas}
x_1^3y_1(x_1+y_1),\,\;x_1y_1^3(x_1+y_1),\,\;\mbox{and} \,\;x_1^5
\end{equation}
as an $\F2$-basis (although $y_1^5$ is in the kernel of $\mathrm{Sq}^1$, it is not a new basis element because of the relation coming from the third element in Lemma~\ref{relacionesmod2}). Now, the first two elements in~(\ref{tresmas}) are the corresponding mod 2 reductions of the two basis elements noted in dimension~5 in Table~\ref{labase}. Therefore the torsion-free summand in $H^5(F(\P^5,2))$ is generated by a class $w_5$ having $x_1^5$ as its mod 2 reduction.

\begin{proposicion}\label{goal1}
The automorphism induced in $\mathbb{Z}$-cohomology by the involution $\rho$ that switches coordinates in $F(\P^5,2)$ is characterized by 
\begin{equation}\label{tresrelacions}
\rho^*(x_2)=y_2,\quad\rho^*(z_3)=z_3,\quad\mbox{and}\quad\rho^*(w_5)=w_5+(x_2+y_2)z_3.
\end{equation}
\end{proposicion}
\begin{proof}
Note that the fibration $V_{6,2}\stackrel{\theta}\to F(\P^5,2)\stackrel{p\,}\to\P^\infty\times\P^\infty$ is $\rho$-equivariant. The first equality in~(\ref{tresrelacions}) is obvious since $x_2$ and $y_2$ ultimately come from the axes in $\P^\infty\times\P^\infty$. The second equality is forced since $H^3(\P^5,2)=\Z2$. For the third equality we necessarily have $\rho^*(w_5)=\varepsilon w_5+\delta_1x_2z_3 + \delta_2y_2z_3$ with $\varepsilon=\pm1$ and $\delta_i\in\{0,1\}$. Since $w_5$ maps nontrivially under the fiber inclusion of $p$, Theorem~\ref{D8actionV} forces $\varepsilon=1$. The fact that $\delta_1=\delta_2=1$ then follows easily by reducing coefficients modulo 2 and using the relation coming from the third generator in Lemma~\ref{relacionesmod2}.
\end{proof}

The $E_2^{p,q}$-term in the CLSS of the involution $\rho$ in Proposition~\ref{goal1} can now be obtained from standard calculations. The result, recorded in Corollary~\ref{e2} below, is depicted in the following chart for $q\leq7$, where a bullet (respectively square, star) stands for a copy of $\Z2$ (respectively $\mathbb{Z}\oplus\Z2$, $\Z2\oplus\Z2$).

{\unitlength=.48mm
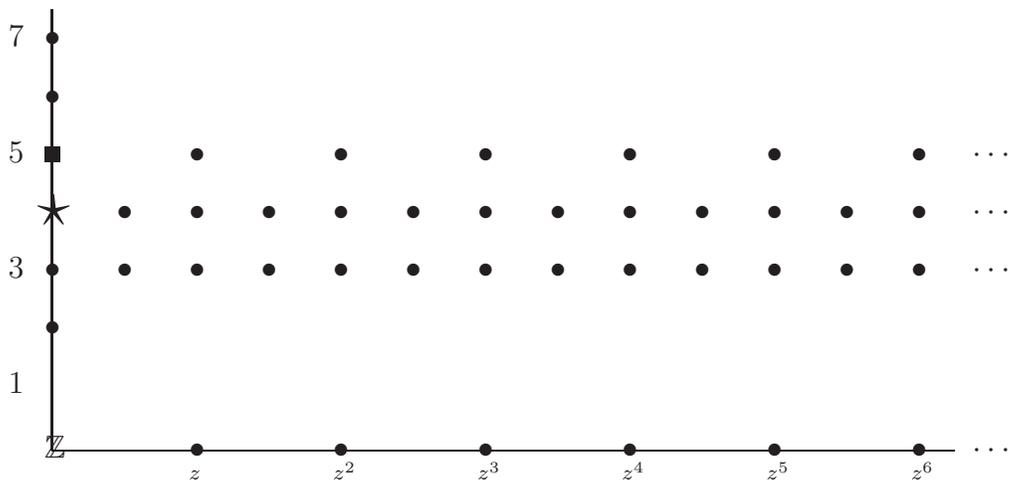
\begin{picture}(0,142)(-27,-12)
\put(0,0){\line(1,0){250}}
\put(0,0){\line(0,1){122}}
\put(38,-8){\scriptsize$z$}
\put(78,-8){\scriptsize$z^2$}
\put(118,-8){\scriptsize$z^3$}
\put(158,-8){\scriptsize$z^4$}
\put(198,-8){\scriptsize$z^5$}
\put(238,-8){\scriptsize$z^6$}
\put(-2,-2){$\mathbb{Z}$}
\multiput(38,-2)(40,0){6}{$\bullet$}
\multiput(-2,32)(20,0){1}{$\bullet$}
\multiput(-2,48)(20,0){13}{$\bullet$}
\multiput(18,64)(20,0){12}{$\bullet$}
\put(-4,62.25){\huge $\star$}
\multiput(38,80)(40,0){6}{$\bullet$}
\put(-2,80){\rule{2mm}{2mm}}
\multiput(-2,96)(20,0){1}{$\bullet$}
\multiput(-2,112)(20,0){1}{$\bullet$}
\put(-12,16){1}
\put(-12,48){3}
\put(-12,80){5}
\put(-12,112){7}
\put(255,-2){$\cdots$}
\put(255,48){$\cdots$}
\put(255,64){$\cdots$}
\put(255,80){$\cdots$}
\end{picture}}

\begin{corolario}\label{e2} Let $H^i$ stand for $H^i(F(\P^5,2))$ as a $\Z2$-module via the action of $\rho^*$. Then:
\begin{itemize}
\item[{\em 1.}] $H^0=\mathbb{Z}$, a trivial $\Z2$-module, so
$H^*(\P^\infty;H^0)=\mathbb{Z}[z]/2z$, $\deg(z)=2$.
\item[{\em 2.}] $H^1=0$, so $H^*(\P^\infty;H^1)=0.$
\item[{\em 3.}] For $i=2,6,7$, $H^i=\Z2[\Z2]$, so $H^*(\P^\infty;H^i)=\Z2$ concentrated in degree $0$.
\item[{\em 4.}] $H^3=\Z2$, so $H^*(\P^\infty;H^3)=\Z2[x]$, $\deg(x)=1$.
\item[{\em 5.}] $H^4=\Z2\oplus\Z2[\Z2]$, so $H^*(\P^\infty;H^4)=\Z2[x]\oplus\Z2$ where the second summand is concentrated in degree $0$.
\item[{\em 6.}] $H^5=\mathbb{Z}\oplus\Z2[\Z2]$ where the additive (torsion) subgroup is in fact a $\Z2$-submodule, but $\rho^*(w_5)=w_5+(1+\rho^*)g$ {\em({\em$w_5$ generates $\mathbb{Z}$, and $g$ generates the $\Z2$-module $\Z2[\Z2]$})}, so
$$H^*(\P^\infty;H^5)=\begin{cases} \mathbb{Z}\oplus\Z2, & *=0; \\ \Z2, & *=2a,\;a>0; \\ 0, & \mbox{otherwise}.\end{cases}$$
\end{itemize}
\end{corolario}

\begin{nota}\label{estructuramultiplicativaNOT}{\em
With respect to the multiplicative structure of the CLSS, a standard cohomology calculation gives that $z\in E_2^{2,0}$ acts injectively on $H^*(\P^\infty;H^3)$, on positive dimensions of $H^*(\P^\infty;H^4)$ and $H^*(\P^\infty;H^5)$, and on the torsion subgroup of $E_2^{0,5}=\mathbb{Z}\oplus\Z2$. Furthermore, for the purpose of the CLSS analysis in the proof of Theorem~\ref{altura}, we will {\it choose} a non-torsion generator in $E_2^{0,5}$ so that all of its $z^i$-multiples are nonzero.
}\end{nota}

\begin{proof}[Proof of Theorem~{\em\ref{altura}}] The generator of $E_2^{0,3}$ must be a permanent cycle since, in view of Theorem~\ref{descripciondesaordenada}, $H^3(B(\P^5,2))=\Z2$---in the sequel we will refer to this sort of argument as ``by convergence''. Since there is no nontrivial target for the $d_4$-differential on the generator of $E_2^{1,3}$, the multiplicative structure of the spectral sequence (Remark~\ref{estructuramultiplicativaNOT}) shows that the whole $(q=3)$-line consists of permanent cycles. This leaves three differentials, originating at nodes 
\begin{equation}\label{tresdiffs4qed}
(3,4),\quad(2,5),\quad\mbox{and}\quad(0,7),
\end{equation}
possibly hitting $z^4\in E_2^{8,0}$. The proof will be complete once we show that $z^4$ is not hit by any of these differentials. 

\medskip
By convergence, all of $E_2^{0,4}$ consists of permanent cycles. One of these elements is given by the $\rho^*$-invariant element $x_2^2+y_2^2$ (see Table~\ref{labase}). Since the permanent cycle in $E_2^{0,3}$ is given by the $\rho^*$-invariant element $z_3$, the product $(x_2^2+y_2^2)z_3$---giving the generator of $E_2^{0,7}$---is a permanent cycle too. This accounts for the $d_8$-differential in~(\ref{tresdiffs4qed}). A second conclusion we draw at this point is that the survival of all of $E_2^{0,4}$ in the spectral sequence implies (in view Remark~\ref{estructuramultiplicativaNOT}) that elements of even total degree in the $(q=4)$-line are also permanent cycles.

\medskip
Before analyzing the two remaining differentials potentially hitting $z^4$, we deduce a few more permanent cycles in the spectral sequence. Firstly, we have observed that $z_3$ gives the generator in $E_2^{0,3}$; now the third relation in~(\ref{relacionesenteras}) shows that the generator in $E_2^{0,6}$ is a permanent cycle. Secondly, since $x_2+y_2$---the generator in $E_2^{0,2}$---is a permanent cycle (say by convergence), $(x_2+y_2)z_3$---the generator of the torsion element in $E_2^{0,5}$---is another permanent cycle. Lastly, Remark~\ref{estructuramultiplicativaNOT} implies that all torsion elements in the $(q=5)$-line are also permanent cycles. Of course, the last assertion accounts for the $d_6$-differential in~(\ref{tresdiffs4qed}).

\medskip
So far we have proved that, in the range shown in the chart, the only elements potentially supporting a nonzero differential are (a) the torsion-free generator in $E_2^{0,5}$---chosen in Remark~\ref{estructuramultiplicativaNOT}---and the elements in the ($q=4$)-line having odd total degree. We next argue that there must be a nonzero $d_k$-differential (with $k\in\{2,5\}$) originating at node $(1,4)$. Indeed, at the start of the spectral sequence there are three nonzero homogeneous torsion elements in total degree 5, however by convergence there are only two such elements in the $E_\infty$-term; the extra element must be the source of a nonzero differential (recall that $E_2^{0,4}$ consists of permanent cycles). But our analysis of permanent cycles shows that such a differential can originate only at node $(1,4)$, as asserted. Now, if the $d_2$-differential originating at node $(1,4)$ is the one that is nonzero, then Remark~\ref{estructuramultiplicativaNOT} implies that this differential repeats horizontally every two degrees, killing in particular the element at node $(3,4)$ and, therefore, accounting for the remaining $d_5$-differential in~(\ref{tresdiffs4qed}).

\medskip
The proof is concluded by drawing a contradiction from the assumption that the $d_2$-differential originating at node $(1,4)$ vanishes. Indeed, such an hypothesis, Remark~\ref{estructuramultiplicativaNOT} and our analysis of permanent cycles would imply, on the one hand, that all $d_2$-differentials originating at the $(q=4)$-line must vanish and, on the other, that the torsion-free generator in $E_2^{0,5}$ (chosen in Remark~\ref{estructuramultiplicativaNOT}) is a $d_\ell$-cycle for $\ell=2,3$. In turn, this situation would imply that the permanent cycles at nodes $(3,3)$ and $(2,4)$ are not killed by any differential. Since this is also the case for the permanent cycle at node $(0,6)$, we would have identified three nonzero torsion homogeneous elements in total degree 6 in the $E_\infty$-term. But this is impossible by convergence.
\end{proof}

\vspace{1.5mm}
Jes\'us Gonz\'alez\quad {\tt jesus@math.cinvestav.mx}

{\sl Departamento de Matem\'aticas, CINVESTAV--IPN

M\'exico City 07000, M\'exico}

\vspace{2mm}

Peter Landweber\quad {\tt landwebe@math.rutgers.edu}

{\sl Department of Mathematics, Rutgers University

Piscataway, NJ 08854, USA}


\begin{thebibliography}{10}

\bibitem{AM}
A.~Adem and R.~J.~Milgram,
{\it Cohomology of Finite Groups,} second edition. 
Grund\-lehren der mathematischen Wissenschaften, 309. 
Springer-Verlag, Berlin, 2004.

\bibitem{barden}
D.~Barden,
``Simply connected five-manifolds'',
{\em Ann. of Math. (2)} {\bf 82} (1965) 365--385.

\bibitem{bausum}
D.~R.~Bausum,
``Embeddings and immersions of manifolds in Euclidean space'',
{\em Trans. Amer. Math. Soc.} {\bf 213} (1975), 263--303. 

\bibitem{electron}
P.~V.~M.~Blagojevi\'c and G.~M.~Ziegler,
``Tetrahedra on deformed spheres and integral group cohomology'',
{\em Electron. J. Combin.} {\bf 16} No.~2 (2009) Research Paper 16, 11 pp.

\bibitem{idealvalued}
P.~V.~M.~Blagojevi\'c and G.~M.~Ziegler,
``The ideal-valued index for a dihedral group action, 
and mass partition by two hyperplanes'', arXiv:0704.1943v4 [math.AT] 
(a shortened version will appear in {\it Topology and its Applications}).

\bibitem{cartan}
H.~Cartan, ``Espaces avec groupes d'op\'erateurs.~I:~Notions 
pr\'eliminaires;~II: La suite spectrale;~applications'',
{\em S\'eminaire Henri Cartan}, tome 3, expos\'es 11 (1--11) and 12 (1--10)
(1950-1951), both available at http://www.numdam.org.

\bibitem{tablas}
D.~M.~Davis,
``Table of immersions and embeddings of real projective spaces'',
available at http://www.lehigh.edu/\mbox{$\sim$}dmd1/immtable.

\bibitem{ES}
D.~Epstein and R.~Schwarzenberger, 
``Imbeddings of real projective spaces'',
{\em Annals of Math.~(2)} {\bf 76} (1962) 180-184.


\bibitem{FH}
E.~Fadell and S.~Husseini, 
``An ideal-valued cohomological index theory with applications to 
Borsuk-Ulam and Bourgin-Yang theorems'', 
{\em Ergod. Th. and Dynam. Sys.} {\bf 8$^*$} (1988) 73-85.

\bibitem{F1} M.~Farber, 
``Topological complexity of motion planning'', 
{\em Discrete Comput. Geom.} {\bf 29} (2003) 211--221.

\bibitem{FGsymm}
M.~Farber and M.~Grant,
``Symmetric motion planning'',
in Topology and Robotics, {\em Contemp. Math.} {\bf 438}, 
Amer. Math. Soc., Providence, RI (2007) 85--104.

\bibitem{FTY}
M.~Farber, S.~Tabachnikov, and S.~Yuzvinsky,
``Topological robotics: motion planning in projective spaces'',
{\em  Int.~Math.~Res.~Not.} {\bf 34} (2003) 1853--1870.

\bibitem{feder}
S.~Feder,
``The reduced symmetric product of projective spaces 
and the generalized Whitney theorem'',
{\em Illinois J. Math.} {\bf 16} (1972) 323--329. 

\bibitem{FZ} E.~M.~Feichtner and G.~M.~Ziegler,
``The integral cohomology algebras of ordered configuration spaces of spheres'',
{\em Doc. Math.} {\bf 5} (2000) 115--139.

\bibitem{FZ02} E.~M.~Feichtner and G.~M.~Ziegler,
``On orbit configuration spaces of spheres'',
{\em Topology Appl.} {\bf 118} (2002) 85--102. 

\bibitem{FeTa} Y.~F\'elix and D.~Tanr\'e,
``The cohomology algebra of unordered configuration spaces'',
{\em J.~London Math.~Soc.} {\bf 72} (2005) 525--544.

\bibitem{taylor}
J.~Gonz\'alez,
``Symmetric topological complexity as the first
obstruction in Goodwillie's Euclidean embedding
tower for real projective spaces'', to appear in {\em
Trans. Amer.~Math.~Soc.} (currently available at
arXiv:0911.1116v4 [math.AT]).

\bibitem{symmotion}J.~Gonz\'alez and P.~Landweber, 
``Symmetric topological complexity of projective and lens spaces'',
{\em Algebr.~Geom.~Topol.} {\bf 9} (2009) 473--494.

\bibitem{v1}J.~Gonz\'alez and P.~Landweber, 
``The integral cohomology groups of configuration spaces of pairs 
of points in real projective spaces'', initial version of 
the present paper available at arXiv:1004.0746v1 [math.AT].

\bibitem{grunbaum}
B.~Gr\"unbaum,
``Partitions of mass-distributions and of convex bodies by hyperplanes'',
{\em Pacific J. Math.} {\bf 10} (1960) 1257--1261. 

\bibitem{handel68}
D.~Handel,
``An embedding theorem for real projective spaces'',
{\em Topology} {\bf 7} (1968) 125--130. 

\bibitem{handeltohoku}
D.~Handel,
``On products in the cohomology of the dihedral groups'',
{\em T\^ohoku Math. J. (2)} {\bf 45} (1993) 13--42.

\bibitem{Ha}
W.~Hantzsche, 
``Einlagerung von Mannigfaltigkeiten in euklidische R\"aume'', 
{\em Math. Zeit.} {\bf 43} (1938) 38-58.

\bibitem{hatcher}
A.~Hatcher,
{\it Algebraic Topology.} 
Cambridge University Press, Cambridge, 2002.

\bibitem{hopf}
H.~Hopf, ``Systeme symmetrischer Bilinearformen und euklidische Modelle
der projektiven R\"aume'', Vierteljschr.~Naturforsch.~Gesellschaft Z\"urich
{\bf 85} (1940) 165-177.

\bibitem{SadokCohenFest}
S.~Kallel, 
``Symmetric products, duality and homological dimension of configuration 
spaces'',  {\em Geom. Topol. Monogr.}, {\bf 13} (2008) 499--527.

\bibitem{KM}
M.~A.~Kervaire and J.~W.~Milnor,
``Groups of homotopy spheres:~I'',
{\em Ann. of Math. (2)} {\bf 77} (1963) 504--537. 

\bibitem{lai}
H.~F.~Lai,
``On the topology of the even-dimensional complex quadrics'',
{\em Proc. Amer. Math. Soc.} {\bf 46} (1974) 419--425. 

\bibitem{LR}
L.~L.~Larmore and R.~D.~Rigdon, 
``Enumerating immersions and embeddings of projective spaces'',
{\em Pacific J. Math.} {\bf 64} (1976) 471--492. 

\bibitem{mahowald}
M.~Mahowald, ``On the embeddability of the real projective spaces'',
{\em Proc.~Amer. Math.~Soc.} {\bf 13} (1962) 763-764.

\bibitem{M64}
M.~Mahowald, 
``On obstruction theory in orientable fibre bundles'', 
{\em Trans. Amer. Math. Soc.} {\bf 110} (1964) 315-349.

\bibitem{mccleary}
J.~McCleary,
{\it A User's Guide to Spectral Sequences}, 
second edition. Cambridge Studies in Advanced Mathematics 58. 
Cambridge University Press, Cambridge, 2001.

\bibitem{MS}
J.~W.~Milnor and J.~D.~Stasheff,
{\it Characteristic Classes.}
Annals of Mathematics Studies No. 76, Princeton University Press,
Princeton NJ, 1974.

\bibitem{munkres}
J.~R.~Munkres,
{\it Elements of Algebraic Topology.} 
Addison-Wesley Publishing Company, Menlo Park, CA, 1984.

\bibitem{prasolov}
V.~V.~Prasolov,
{\it Elements of homology theory.}
Translated from the 2005 Russian original by Olga Sipacheva. 
Graduate Studies in Mathematics 81. AMS, Providence, RI, 2007.

\bibitem{ranicki}
A.~Ranicki,
{\it Algebraic and Geometric Surgery.}
Oxford Mathematical Monographs, Oxford Science Publications. 
Oxford University Press, 2002. 
Electronic version (August 2009) available at
http://www.maths.ed.ac.uk/$\sim$aar/books/surgery.pdf.

\bibitem{reesodd}
E.~Rees, ``Embedding odd torsion manifolds'', 
{\em Bull.~London Math.~Soc.} {\bf 3} (1971) 356--362.

\bibitem{rees}
E.~Rees, ``Embeddings of real projective spaces'', 
{\em Topology} {\bf 10} (1971) 309--312.


\bibitem{ST}
H.~Seifert and W.~Threlfall,
{\it A Textbook of Topology,}
translated from the German 1934 edition by Michael A.~Goldman, 
with a preface by Joan S.~Birman. 
Pure and Applied Mathematics, 89. Academic Press, Inc. New York-London, 1980. 

\bibitem{sutherland}
W.~A.~Sutherland, 
``A note on the parallelizability of sphere-bundles over spheres'',
{\em J. London Math. Soc.} {\bf 39} (1964) 55--62. 

\bibitem{collins}
P.~Teichner,
{\it Slice Knots:~Knot Theory in the $4^{\mathrm th}$ Dimension.}
Lecture notes by Julia Collins and Mark Powell.
Electronic version (October 2009) available at
http://www.maths.ed.ac.uk/$\sim$s0681349/\#research.

\bibitem{Totaro}
B.~Totaro, ``Configuration spaces of algebraic varieties'',
{\em Topology\hspace{.2mm}} {\bf 35} (1996) 1057--1067.

\bibitem{wall}
C.~T.~C.~Wall,
``Killing the middle homotopy groups of odd dimensional manifolds'',
{\em Trans. Amer. Math. Soc.} {\bf 103} (1962) 421--433. 

\bibitem{whitehead}
G.~W.~Whitehead,
{\it Elements of Homotopy Theory.}
Graduate Texts in Mathematics, 61. Springer-Verlag, New York-Berlin, 1978.

\end{thebibliography}
\end{document}